\documentclass[amstex,11pt]{amsart}
\usepackage{geometry}
\usepackage{mathtools,amssymb,amsthm,mathrsfs,color,lineno,paralist,graphicx,float}
\usepackage[backref=page,colorlinks,
linkcolor=blue,
anchorcolor=red,
citecolor=red, 
]{hyperref}

\usepackage[T1]{fontenc}
\usepackage[utf8]{inputenc}
\usepackage{mathrsfs}
\usepackage{doi}
\usepackage{cleveref}
\usepackage{calc}
\definecolor{bleu1}{RGB}{0,57,128}
\def\bleu1{\color{bleu1}}
\usepackage{etoolbox}
\patchcmd{\section}{\normalfont}{\normalfont \bleu1}{}{}
\patchcmd{\subsection}{\normalfont}{\normalfont \bleu1}{}{}
\patchcmd{\subsubsection}{\normalfont}{\normalfont \bleu1}{}{}

\newtheorem{thm}{\bleu1Theorem}[section]
\newtheorem{prop}[thm]{\bleu1 Proposition}
\theoremstyle{plain}
\theoremstyle{definition}

\newtheorem{lem}[thm]{\bleu1 Lemma}
\newtheorem{defn}[thm]{\bleu1 Definition}
\newtheorem{claim}[thm]{\bleu1 Claim}

\newtheorem{rem}[thm]{\bleu1 Remark}

\numberwithin{equation}{section} 

\newcommand{\z}{\mathbb{Z}}
\newcommand{\ba}{\mathbf{a}}

\newcommand{\supp}{\mathrm{supp}}

\definecolor{lgray}{gray}{0.9}

\def\leq{\leqslant}
\def\geq{\geqslant}
\def\tilde{\widetilde}
\def\hat{\widehat}

\def\lk {\left}
\def\re {\right}
\allowdisplaybreaks
\renewcommand*{\backref}[1]{}
\renewcommand*{\backrefalt}[4]{\quad \tiny
  \ifcase #1 (\textbf{NOT CITED.})
  \or    (Cited on page~#2.)
  \else   (Cited on page~#2.)
  \fi}
\makeatletter
\@namedef{subjclassname@2020}{\textup{2020} Mathematics Subject Classification}
\subjclass[2020]{37H05, 37B40.}

\title[Mean Li-Yorke chaos for random dynamical systems]{Mean Li-Yorke chaos along any infinite sequence for  infinite-dimensional  random dynamical systems}
\author{Chunlin Liu}
\address[Chunlin Liu]{CAS Wu Wen-Tsun Key Laboratory of Mathematics\\School of Mathematical Sciences\\ University of Science and Technology of China\\ Hefei, Anhui, 230026, P.R. China}
\email{lcl666@mail.ustc.edu.cn}
\author{Feng Tan}
\address[Feng Tan]{School of  Mathematical Sciences\\ South China Normal University \\  Guangzhou, GuangDong, 510631, P.R. China}
\email{tanfeng@scnu.edu.cn}

\author{Jianhua Zhang}
\address[Jianhua Zhang]{CAS Wu Wen-Tsun Key Laboratory of Mathematics\\School of Mathematical Sciences\\ University of Science and Technology of China\\ Hefei, Anhui, 230026, P.R. China}
\email{leapforg@mail.ustc.edu.cn}

\keywords{Infinite-dimensional random dynamical systems;  Mean Li-Yorke chaos; Positive entropy.}
\date{\today}
\subjclass[2020]{37H05, 37B40.}

\begin{document}

\begin{abstract}
 In this paper, we  study the mean Li-Yorke chaotic phenomenon along any infinite positive integer sequence for infinite-dimensional random dynamical systems. To be precise, we prove that if an injective continuous infinite-dimensional random dynamical system $(X,\phi)$ over an invertible ergodic Polish system $(\Omega,\mathcal{F},\mathbb{P},\theta)$ admits a $\phi$-invariant random compact  set $K$ with positive topological entropy, then given a positive integer sequence $\ba=\{a_i\}_{i\in\mathbb{N}}$ with $\lim_{i\to+\infty}a_i=+\infty$,   for $\mathbb{P}$-a.s. $\omega\in\Omega$ there  exists an  uncountable subset $S(\omega)\subset K(\omega)$ and $\epsilon(\omega)>0$ such that  for any two distinct points $x_1$, $x_2\in S(\omega)$ with following properties
  \begin{align*}
\liminf_{N\to+\infty}\frac{1}{N}\sum_{i=1}^{N} d\big(\phi(a_i, \omega)x_1, \phi(a_i, \omega)x_2\big)=0,\quad\limsup_{N\to+\infty}\frac{1}{N}\sum_{i=1}^{N} d\big(\phi(a_i, \omega)x_1, \phi(a_i, \omega)x_2\big)>\epsilon(\omega),
\end{align*}
where $d$ is a compatible complete metric on $X$.
%Furthermore, if $(X,\phi)$ is only a continuous infinite-dimensional random dynamical system, then 
\end{abstract}
\maketitle

\section{Introduction}

 %Entropy as a invariant of the measure-preserving systems,  quantifies the  complex behavior of a measure-preserving system. In \cite{MR0103254}, Kolmogorov firstly introduced the metric entropy  of  the measure-preserving systems. After that, Adler, Konheim and McAndrew defined the topological entropy of the topological systems   in \cite{MR175106}. 

In the study of  dynamical systems, there is a fundamental question: how to describe the chaotic phenomenon.
In the past few decades, there are various results about this question, such as \cite{MR175106,MR2101573,MR1046376,MR213508,MR1874089,MR1146241,MR385028, MR1227094,MR0182020}. The reader can refer to \cite{MR3431162} for more aspects and  details.  In this paper, we focus on the 
relation between the positive entropy and  the mean Li-Yorke chaos in infinite-dimensional random dynamical systems.
In deterministic finite-dimensional dynamical systems,  Blanchard et al. \cite{BGKM} proved that the  positive entropy implies the Li-Yorke chaos; moreover, Downarowicz \cite{MR3119189} observed that the mean Li-Yorke chaos is equivalent to the  DC2 chaos and proved that the positive  entropy implies the mean Li-Yorke chaos.  For further study on the Li-Yorke chaos, see  \cite{BeBP,HLY,MR3788230,liu2022pinsker,WCLZ,WCZ}.

 %Finite-dimensional random dynamical systems and infinite-dimensional random dynamical systems are usually generated by some classes of  stochastic ordinary differential equations and stochastic partial differential equations, respectively (see \cite{MR600452,MR608026,MR876080,MR2393626,MR553909,MR722131} for example).  
 Since about 1980,  random dynamical systems have attracted widespread attention. There is an extensive literature on random dynamical systems, for example \cite{A,K86,KiL,LY881,LY882,MR1369243,MR647807}. However,  little is known about the relation between the entropy and the Li-Yorke chaos  for  random dynamical systems. 
Until 2017,   the first result was obtained by Huang and Lu \cite{HL}. They showed that the positive entropy  implies the Li-Yorke chaos for almost sure fibers  in infinite-dimensional random dynamical systems. Recently, the second author \cite{Tan} proved 
 that the positive entropy  implies the DC2 chaos for almost sure fibers.
%There is a vast amount of work in random dynamical systems on ergodic theory (\), the existence of random attractors (\cite{MR1451294,schuma}), the theory of invariant manifolds (\cite{MR647807,MR2674952}) and so on. 
If the phase space is not compact, even though  two points are very close under iteration of a  map for most of the time, then the average of  distance of these two points under the iteration  is not necessarily small, since these two points may be far apart  under the iteration  for  the rest time. 
 Therefore,  the DC2 chaos may not imply the mean Li-Yorke chaos in infinite-dimensional random dynamical systems which is different from the case in  deterministic finite-dimensional dynamical systems.   

It is natural to ask  that whether the positive entropy implies  the mean Li-Yorke chaos in infinite-dimensional random dynamical systems. In this paper, we give a positive answer. The main difficulty in our proofs is  that  the phase space 
is not compact, even  not locally compact. This leads to the invalidation of many classical results in compact systems. By utilizing the regularity of probability measure on the Polish measurable space,  we show that there also are many asymptotic pairs  in non-compact systems (see Theorem \ref{infinity}). Besides, 
 we prove  that  the set of the  pairs  being separate in the mean sense along a given infinite positive integer sequence    has the full measure   under  the  product of  conditional measure with respect to the relative Pinsker factor in non-compact 
  systems with  positive entropy (see Theorem \ref{22-03-06-05}). Through these two theorems,  we show that the positive entropy  implies the mean Li-Yorke chaos for almost sure fibers in infinite-dimensional random dynamical systems.
%extend many results for compact spaces to noncompact spaces  and prove that relative Pinsker algebra of noncompact Lebesgue system is characteristic $\sigma$-algebra for any infinite integer sequence (see Proposition \ref{21-03-06-02}) to overcome the difficulty above and prove that positive entropy implies mean Li-Yorke chaos in infinite-dimensional random dynamical systems.
 
To describe  the results of this paper precisely,  we present the basic setting. Assume that $(X,d)$ is a  complete separable metric space and $(\Omega,\mathcal{F},\mathbb{P},\theta)$ is an invertible ergodic Polish system (see Section \ref{2.1} for the  definition).  Consider the discrete random dynamical system over $(\Omega,\mathcal{F},\mathbb{P},\theta)$ defined by the measurable map 
  $$\phi:\mathbb{N}_0\times\Omega\times X\to X,\quad (n,\omega, x)\mapsto\phi(n, \omega, x)$$
where $\mathbb{N}_0=\mathbb{N}\cup\{0\}$, such that the map $\phi(n,\omega):=\phi(n,\omega,\cdot)$ for any  $\omega\in\Omega$ satisfies  the following properties: 
  \begin{enumerate}[(i)]
  \item $\phi(0,\omega)=\text{Id}_X$;
  \item $\phi(n+m,\omega)=\phi(n,\theta^m\omega)\circ \phi(m,\omega)$ for any $n,m\in\mathbb{N}_0$.
  \end{enumerate}
The pair $(X,\phi)$ is called a \emph{continuous (resp. an injective continuous) random dynamical system} over $(\Omega,\mathcal{F},\mathbb{P},\theta)$ if $\phi(n,\omega)$ is a   continuous (resp. an injective continuous) from $X$ to itself for any $n\in\mathbb{N}$ and $\mathbb{P}$-a.s. $\omega\in\Omega$.
  
   A multifunction $K=\{K(\omega)\}_{\omega\in\Omega}$  is called a \emph{random closed (resp. compact) set} if  $\mathbb{P}$-a.s. $\omega\in\Omega$, $K(\omega)$  is a nonempty closed (resp. compact) set of $X$ and 
  $$\omega\mapsto\inf_{y\in K(\omega)}d(x,y)$$
 is  measurable for any $x\in X$.   In this paper, $K$ will also be  regarded as  a subset of $\Omega\times X$ by $\{(\omega,x): x\in K(\omega)\}$.  A random closed set $K$ is said to be $\phi$-invariant  if  $\phi(n, \omega)K(\omega)=K(\theta^n\omega)$ for  any $n\in\mathbb{N}$ and $\omega\in\Omega$.
\iffalse
\begin{rem}
In the above,  the assumption of `for any  $\omega$'  is just convenient to our proof. In fact,  `for any  $\omega$' can be  changed to `for $\mathbb{P}$-a.s. $\omega$'  by using the  measurable isomorphism (see \cite[Theorem 2.8]{G}). 
\end{rem}
\fi

The first main result of this paper is stated as follows.
\begin{thm}\label{main2}
	Let $(X,\phi)$ be  a  continuous random dynamical systems   over an invertible ergodic  Polish system $(\Omega,\mathcal{F},\mathbb{P},\theta)$ and  $K\subset \Omega\times X$ be a $\phi$-invariant random compact  set. 
	 If the topological entropy  $h_{top}(\phi, K)>0$ of $(K,\phi)$ (see Section \ref{22-11-28-01}), then $(K,\phi)$ is mean Li-Yorke chaotic for almost sure fibers. Namely,  for $\mathbb{P}$-a.s. $\omega\in\Omega$ there exists a Mycielski set (i.e., a union of countably many Cantor sets)  $S(\omega)\subset K(\omega)$ and $\epsilon(\omega)>0$  such that for any distinct points $x_1$, $x_2\in S(\omega)$ with following properties
  \begin{align*}
&\liminf_{N\to+\infty}\frac{1}{N}\sum_{i=1}^{N} d\big(\phi(i, \omega)x_1, \phi(i, \omega)x_2\big)=0,
\\&\limsup_{N\to+\infty}\frac{1}{N}\sum_{i=1}^{N} d\big(\phi(i, \omega)x_1, \phi(i, \omega)x_2\big)>\epsilon(\omega).
\end{align*}
\end{thm}
Furthermore, the  Li-Yorke  chaos along a given sequence has been investigated in the deterministic dynamical systems. For example,  Huang, Li and Ye \cite{HLY1} showed that the positive entropy implies the chaos along any infinite sequence, and 
the reader can see \cite{MR3788230,liu2022pinsker} for more results. Inspired by the above works,   we prove that an injective continuous infinite-dimensional random dynamical systems with positive  entropy is mean Li-Yorke chaotic along any infinite sequence for almost sure fibers. Specifically, 
\begin{thm}\label{main1}
Let $(X,\phi)$ be an injective  continuous random dynamical systems over an invertible ergodic Polish system $(\Omega,\mathcal{F},\mathbb{P},\theta)$ and  $K\subset \Omega\times X$ be a $\phi$-invariant random compact  set.  
If the topological entropy  $h_{top}(\phi, K)>0$ of $(K,\phi)$, then given an  infinite positive integer sequence  $\ba=\{a_i\}_{i\in\mathbb{N}}$ $($i.e. $\lim_{i\to+\infty}a_i=+\infty)$, $(K,\phi)$ is mean Li-Yorke chaotic along $\ba$ for almost sure fibers. Namely, for $\mathbb{P}$-a.s. $\omega\in\Omega$ there exists a Mycielski set  $S(\omega)\subset K(\omega)$ and $\epsilon(\omega)>0$  such that for any distinct points $x_1$, $x_2\in S(\omega)$ with following properties
	\begin{align}
	\label{22-03-09-01}&\liminf_{N\to+\infty}\frac{1}{N}\sum_{i=1}^{N} d\big(\phi(a_i, \omega)x_1, \phi(a_i, \omega)x_2\big)=0,
	\\\label{22-03-09-02}&\limsup_{N\to+\infty}\frac{1}{N}\sum_{i=1}^{N} d\big(\phi(a_i, \omega)x_1, \phi(a_i, \omega)x_2\big)>\epsilon(\omega).
	\end{align}
\end{thm}

This paper is organized as follows. In Section \ref{22-03-05-01}, we review some necessary notions and required properties. In Section \ref{22-03-08-01},  we prove that  stable sets in  a Polish system are dense in the support of
 conditional measure with respect to  relative Pinsker factor, namely Theorem \ref{infinity}; at last,  we prove the set of all separating pairs in a Polish system with positive entropy is  full measure of  the  product of conditional  measure with respect to the relative Pinsker factor, namely Theorem \ref{22-03-06-05}. In  \Cref{22-03-08-02}, we prove  Theorem \ref{main2} and Theorem \ref{main1}.

\section{Preliminary}\label{22-03-05-01}

%Throughout this paper, we always assume that a probability space $(X,\mathcal{X},\mu)$  is countably generated, i.e. there is a sequence measurable sets $\{A_i\}_{i\in\mathbb{N}}\subset\mathcal{X}$  such that for any measurable set %$A\in\mathcal{X}$ and $\epsilon>0$, there exists $i\in\mathbb{N}$ such that $\mu(A\Delta A_i)\leq\epsilon$.
 In this section, we review some basic concepts and results from the theory of measure-preserving dynamical systems.
\subsection{Conditional measure-theoretic entropy}\label{2.1}%\cite{MR648108}
 In this subsection, we review the definition of  conditional measure-theoretic entropy of a measure-preserving dynamical system and state some properties of  the conditional measure-theoretic entropy. The reader can refer to \cite{ELW,Wal} for more details.
 
\emph{A Polish  space} $X$ means that is a separable topological space whose topology is metrizable by a complete metric. \emph{A Polish probability space} $(X,\mathcal{X},\mu)$ means that $X$  is a Polish space,  $\mathcal{X}$ is the  Borel $\sigma$-algebra,  and $\mu$ is the probability measure on $\mathcal{X}$. In this moment, $(X, \mathcal{X})$ is called \emph{Polish measurable space}.  A probability space $(X ,\mathcal{X},\mu)$  is called a 
\emph{standard probability space} if  $(X,\mathcal{X})$ is  isomorphic to a Polish measurable space. \emph{A measure-preserving dynamical system} (MDS for short)  $(X,\mathcal{X},\mu,T)$ means that  $T$ is a measure-preserving map  on the probability space $(X,\mathcal{X},\mu)$.  A MDS $(X,\mathcal{X},\mu,T)$ is called \emph{ergodic} if for any $A\in\mathcal{X}$, $T^{-1}A=A$ implies that $\mu(A)\mu(X\setminus A)=0$.
 
    Given two MDSs  $(X,\mathcal{X},\mu,T)$ and $(Y,\mathcal{Y},\nu,S)$, we say that  $(Y,\mathcal{Y},\nu,S)$ is \emph{a factor} of $(X,\mathcal{X},\mu,T)$ if there exists a
measure-preserving map $\pi:(X,\mathcal{X},\mu, T)\rightarrow (Y,\mathcal{Y},\nu, S)$ such that $\pi\circ T=S\circ\pi$, and $\pi$ is called a   \emph{factor map}.
\begin{defn}
	\label{22-07-15-01}
A MDS $(X,\mathcal{X},\mu,T)$ is called a \emph{Polish system} if $(X,\mathcal{X},\mu)$ is a Polish probability space. If $T^{-1}: X\to X$ exists and is measurable,  then  $(X,\mathcal{X},\mu,T)$ is called an \emph{invertible} MDS. 
\end{defn}
Let $(X,\mathcal{X},\mu)$ be a probability space and $\mathcal{S}$ be a sub-$\sigma$-algebra of $\mathcal{X}$. \emph{The conditional expectation} is a linear operator
$\mathbb{E}(\cdot |\mathcal{S}): \mathcal{L}^1(X,\mathcal{X},\mu)\rightarrow \mathcal{L}^1(X,\mathcal{S},\mu)$ charactered by following properties:
\begin{enumerate}[(i)]
	\item for every $f \in \mathcal{L}^1(X,\mathcal{X},\mu)$, $\mathbb{E}(f|\mathcal{S})$ is $\mathcal{S}$-measurable;
	\item for any $A\in\mathcal{S}$ and $f \in \mathcal{L}^1(X,\mathcal{X},\mu)$, $\int_Afd\mu=\int_A\mathbb{E}(f|\mathcal{S})d\mu$.
\end{enumerate}

Recall the Martingale theorem. The reader can see it in  \cite[Theorem 14.26]{G} or \cite[Chapter 5.2]{EW}.
\begin{thm} [Martingale theorem]\label{M}
	Given a  probability space $(X,\mathcal{X},\mu)$,  suppose
	that $\{\mathcal{S}_n\}_{n\in\mathbb{N}}$ is a decreasing sequence (resp. an increasing sequence) of sub-$\sigma$-algebra of $\mathcal{X}$ and $\mathcal{S}=\bigcap_{n\geq 1}\mathcal{S}_n$ (resp. $\mathcal{S}=\sigma(\bigvee_{n\geq 1}\mathcal{S}_n)$). Then for any $f\in\mathcal{L}^1(X,\mathcal{X},\mu)$,
	$\mathbb{E}(f|\mathcal{S}_n)\rightarrow \mathbb{E}(f|\mathcal{S})$
	as $n\rightarrow+\infty$ in $\mathcal{L}^1(X,\mathcal{X},\mu)$ and $\mu$-almost sure.
\end{thm}

Let $(X,\mathcal{X}, \mu, T)$ be a MDS. Given a finite measurable partition $\alpha $ and a sub-$\sigma$-algebra $\mathcal{S}$ of $\mathcal{X}$, denote
$$
H_{\mu}(\alpha|\mathcal{S}):=\sum_{A\in \alpha} \int_X-\mathbb{E}(1_A|\mathcal{S}) \log\mathbb{E}(1_A|\mathcal{S}) d \mu,$$ 
Note that $\{H_\mu(\bigvee_{i=0}^{n-1}T^{-i}\alpha|\mathcal{S})\}_{n\in\mathbb{N}_0}$ is a non-negative and sub-additive sequence. Therefore,  the conditional
measure-theoretic entropy of $\mu$ with respect to $\mathcal{S}$ is defined as 
$$h_\mu(T|\mathcal{S}):=\sup_{\alpha} h_\mu(T,\alpha|\mathcal{S}):=\sup_{\alpha}\lim_{n\to+\infty}\frac{1}{n}H_\mu\left(\bigvee_{i=0}^{n-1}T^{-i}\alpha|\mathcal{S}\right),$$
where $\alpha$ runs over all finite measurable partitions of $\mathcal{X}$. Let $\pi:(X,\mathcal{X},\mu,T)\rightarrow (Y,\mathcal{Y},\nu,S)$ be a factor map between two MDSs. The conditional  measure-theoretic entropy  of $\mu$ with respect to $\pi$ is defined as $h_\mu(T|\pi):=h_\mu(T|\pi^{-1}\mathcal{Y}).$ The following result is a generalization of Abramov-Rohlin formula from \cite{BC}.
\begin{lem} \label{GAR} 
	Let $\pi:(X,\mathcal{X},\mu,T)\rightarrow (Y,\mathcal{Y},\nu,S)$ and $\psi:(Y,\mathcal{Y},\nu,S)\rightarrow
	(Z,\mathcal{Z},\eta,R)$ be two factor maps between two MDSs on standard probability spaces. Then  $\psi\circ
	\pi:(X,\mathcal{X},\mu,T)\rightarrow (Z,\mathcal{Z},\eta,R)$ is also a factor map  and
	$$h_\mu(T|\psi\circ \pi)=h_\mu(T|\pi)+h_\nu(S|\psi).$$ \end{lem}

\subsection{Entropy for random dynamical systems}
\label{22-11-28-01}
In this subsection, we mainly  introduce the entropy and the variational principle in random dynamical systems. Throughout this subsection, we assume that $(X,\phi)$ is a continuous random dynamical system over the ergodic Polish system  $(\Omega, \mathcal{F},\mathbb{P},\theta)$, where $(X,d)$ is a complete separable metric space. The 
reader can  refer to \cite{A,KiL} for more details.
\begin{defn} 
Suppose that $(X,\phi)$ is a continuous random dynamical system over an  invertible  ergodic Polish system $(\Omega,\mathcal{F},\mathbb{P},\theta)$. The map
$$\Phi:\Omega\times X\rightarrow \Omega\times X,\ \ (\omega,x)\mapsto(\theta \omega,\phi(1,\omega)x)$$ is said to be a \emph{skew product  system }induced by $(X,\phi)$. 
\end{defn}

Let  $\pi_\Omega:\Omega\times X\rightarrow \Omega$ be the projection. A
probability measure $\mu$ on the measurable space $(\Omega\times X,\mathcal{F}\times  \mathcal{X})$ is said
to have \emph{marginal $\mathbb{P}$}  if $(\pi_{\Omega})_*\mu=\mathbb{P}$, namely $\mu(A\times X)=\mathbb{P}(A)$ for any measurable subset $A\in\mathcal{F}$.  Denote 
$\mathcal{P}_{\mathbb{P}}(\Omega\times X)$  as the collection of such measures, $\mathcal{M}_{\mathbb{P}}(\Omega\times X,\Phi)$ as 
the collections of $\Phi$-invariant elements of $\mathcal{P}_{\mathbb{P}}(\Omega\times X)$ and  $\mathcal{E}_{\mathbb{P}}(\Omega\times X,\Phi)$  as the collections of ergodic elements of $ \mathcal{M}_{\mathbb{P}}(\Omega\times X,\Phi)$. For convenience, we omit $\Phi$ and write $ \mathcal{M}_{\mathbb{P}}(\Omega\times X,\Phi)$ and $\mathcal{E}_{\mathbb{P}}(\Omega\times X,\Phi)$ as $ \mathcal{M}_{\mathbb{P}}(\Omega\times X)$ and $\mathcal{E}_{\mathbb{P}}(\Omega\times X)$, respectively.

Assume that $K$ is a $\phi$-invariant random compact set.  Then  there  exists $\mu\in  \mathcal{M}_{\mathbb{P}}(\Omega\times X)$ with $\mu(K)=1$ (see \cite{C} or \cite[Theorem 1.6.13]{A}). Set
 $$\mathcal{M}^K_{\mathbb{P}}(\Omega\times X)=\{ \mu\in \mathcal{M}_{\mathbb{P}}(\Omega\times X):\mu(K)=1\}.$$
%By \cite[Proposition 1.6.2]{A}, $K$ is a Borel subset of $\Omega\times X$.
% Given  $\mu\in \mathcal{M}^K_{\mathbb{P}}(\Omega\times X)$,  one has that  $(K, \mathcal{X}\upharpoonright_{K},\mu)$ is isomorphic to a  Polish probability space  where  $\mathcal{X}\upharpoonright_{K}:=\{A\cap K: A\in \mathcal{F}\times \mathcal{X}\}$(see \cite[Theorem 2.8]{G}). Thus $(K,\mathcal{X}\upharpoonright_{K},\mu, \Phi)$ is a MDS which is isomorphic to a Polish MDS  , and
%$\pi_\Omega: (K,\mathcal{X}\upharpoonright_{K},\mu, \Phi)\rightarrow (\Omega, \mathcal{F}, \mathbb{P},\theta)$ is a factor map between two MDS. 
Since $K$ is a Borel subset of $\Omega\times X$ (see \cite[Proposition 1.6.2]{A}), $(K,\mathcal{K},\mu,\Phi)$ is a MDS on a standard probability space where $\mathcal{K}=\{A\cap K: A\in\mathcal{F}\times\mathcal{X}\}$ and 
$$\pi_{\Omega}: (K,\mathcal{K},\mu,\Phi)\to (\Omega,\mathcal{F},\mathbb{P},\theta)$$ 
 is a factor map between MDSs on standard probability spaces. The \emph{measure-theoretic entropy} $(K,\phi)$ with respect to $\mu$ is defined by
$$h_\mu(\phi,K):=h_\mu(\Phi|\pi_\Omega)=h_\mu(\Phi|\pi_\Omega^{-1}\mathcal{F}).$$ That is, $h_{\mu}(\phi,K)$ is the conditional measure-theoretic entropy of $(K,\mathcal{K},\mu,\Phi)$ with respect to
$\pi_\Omega$.
% By Bogenschutz \cite{B},
%$$h_\mu(\phi)=\sup\{
%h_\mu\big(\Phi,\pi^{-1}_X\alpha|\pi_\Omega^{-1}\mathcal{F}\big):\alpha \text{ is a finite measurable partition of
%}X\},$$
%where $\pi_X:\Omega\times X\rightarrow X$ is the projection.

 Now we prepare to give the definition of the topological entropy of $(K,\phi)$. The reader can refer to \cite{MR1181382, K01} for the definition of the topological entropy when $X$ is compact. For $\omega\in\Omega, \epsilon>0$ and $n\in\mathbb{N}$, a subset $E$ of $K(\omega)$ is called an $(\omega, n,\epsilon,\phi)$-separated subsets of $K(\omega)$ if for any distinct points $x,y\in E$, one has that 
 \begin{align*}
 \max_{0\leq i\leq n-1} d(\phi(i,\omega)x, \phi(i,\omega)y)>\epsilon.
 \end{align*}
Denote the maximal cardinality of all $(\omega, n,\epsilon,\phi)$-separated subsets of $K(\omega)$ as $r_n(K,\omega,\epsilon,\phi)$. The \emph{topological entropy} of $(K, \phi)$ is defined as following 
\begin{align}
h_{top}(\phi, K):=\lim_{\epsilon\to 0}\limsup_{n\to+\infty}\frac{1}{n}\int_{\Omega}\log r_n(K,\omega,\epsilon,\phi)d\mathbb{P}(\omega).
\end{align}
    
    We end this subsection by presenting the  variational principle  for random dynamical systems.  The reader can see it in \cite[ Proposition 3.7]{HL}.
\begin{prop}
\label{VP} 
Let  $(X,\phi)$ be a continuous random dynamical system over $(\Omega,\mathcal{F},\mathbb{P},\theta)$ and  $K$ be a $\phi$-invariant random  compact set. Then $$h_{\text{top}}(\phi,K)=\sup \{h_\mu(\phi,K):\mu\in
\mathcal{M}_{\mathbb{P}}^K(\Omega\times X)\}=\sup \{h_\mu(\phi,K):\mu\in \mathcal{E}_{\mathbb{P}}^K(\Omega\times X)\},$$ 
where   $ \mathcal{E}^K_{\mathbb{P}}(\Omega\times X)$ is  the set of the ergodic elements of $ \mathcal{M}^K_{\mathbb{P}}(\Omega\times X)$.
\end{prop}

\subsection{Natural extension}\label{Natural}
In this subsection, we review the natural extension of the MDSs on standard probability spaces. Assume that $(X,\mathcal{X},\mu,T)$ is  a MDS on a standard probability space.    Let
\begin{align*}
&\bar{X}={\{\vec{x}=(x_{i})_{i\in  \mathbb{Z}}\in X^{\mathbb{Z}}: \, Tx_{i}=x_{{i+1}}, i\in \mathbb{Z}\}},
\\\quad&\bar{T}:\bar{X}\to \bar{X},\quad (x_{i})_{i\in \mathbb{Z}}\mapsto(Tx_{{i}})_{i\in \mathbb{Z}},
\end{align*} 
$\bar{\mathcal{X}}$ be the $\sigma$-algebra which is  generated by $\bigcup_{n\in\mathbb{Z}} \Pi_{n,X}^{-1}\mathcal{X}$  where $\Pi_{n,X}: \bar{X}\to X $ with $\Pi_{n, X}(\vec{x})=x_{n}$,  and   $\bar{\mu}$ be the measure on $\bar{\mathcal{X}}$ which is defined by  $\bar{\mu}(\Pi_{n,X}^{-1}(A))=\mu(A)$  for $A\in
\mathcal{X}$.  It is clear that $(\bar{X}, \bar{\mathcal{X}},\bar{\mu}, \bar{T})$ is an  invertible MDS  on a standard probability space.  Then  
$$\Pi_X:=\Pi_{0,X}:(\bar{X},\bar{\mathcal{X}},\bar{\mu},\bar{T})\rightarrow(X,\mathcal{X},\mu,T)$$ is a factor map  and $(\bar{X},\bar{\mathcal{X}},\bar{\mu},\bar{T})$ is  called the natural extension of $(X,\mathcal{X},\mu,T)$. The reader can refer to \cite{Roh2} for the proof. In \cite{Roh1}, it is proved that $(\bar{X},\bar{\mathcal{X}},\bar{\mu},\bar{T})$ is ergodic if and only if $(X,\mathcal{X},\mu,T)$ is ergodic.

%Note that  $(\bar{X}, \bar{\mathcal{X}},\bar{\mu}, \bar{T})$ may not be a Polish system, because $\bar{X}$ is only a Borel measurable subset of the Polish space $X^\mathbb{Z}$.

Now, we state  a result about the conditional measure-theoretic entropy of the  natural extension from \cite[Lemma 3.2]{HL}. 
  \begin{lem} \label{na=zero} 
 Let $\Pi_X:(\bar{X},\bar{\mathcal{X}},\bar{\mu},\bar{T})\rightarrow
(X,\mathcal{X},\mu,T)$ be  the natural extension  of the MDS $(X,\mathcal{X},\mu,T)$ on a standard probability space. Then
$h_{\bar{\mu}}(\bar{T}|\Pi_X)=0$.
\end{lem}

\subsection{Disintegration of measures}
\label{22-07-06-01}
In this subsection, we recall some notations and results on the disintegration of measures which are summarized from \cite[Chapter 5 and 6]{EW}.

For a factor map $\pi:(X,\mathcal{X},\mu,T)\rightarrow (Y,\mathcal{Y},\nu,S)$  between two MDSs on standard probability spaces, there is a set of conditional probability measures $\{\mu_y\}_{y\in Y}$ with the following
properties: 
\begin{itemize}
	\item $\mu_y$ is a probability measure on $(X,\mathcal{X})$ with $\mu_y(\pi^{-1}(y))=1$ for $\nu$-a.s. $y\in Y$;
	\item for each $f \in \mathcal{L}^1(X,\mathcal{X},\mu)$,  one  has $f \in \mathcal{L}^1(X,\mathcal{X},\mu_y)$ for $\nu$-a.s. $y\in Y$, the map $y \mapsto
	\int_X f\,d\mu_y$ is in $\mathcal{L}^1(Y,\mathcal{Y},\nu)$ and $\int_Y \left(\int_X f\,d\mu_y \right)\, d\nu(y)=\int_X f \,d\mu$.
\end{itemize}
$\mu=\int_Y \mu_y d \nu(y)$ is called the \emph{disintegration of $\mu$ relative to
	the  factor $(Y,\mathcal{Y},\nu,S)$}.
	Furthermore, the measures $\{\mu_y\}_{y\in Y}$ are essentially unique and $T_*\mu_y=\mu_{Sy}$  for $\nu$-a.s. $y\in Y$. The conditional  expectations and the conditional measures are related by
\begin{equation} \label{meas3} \mathbb{E}(f|\pi^{-1}\mathcal{Y})(x)=\int_X f\,d\mu_{\pi(x)} \ \ \text{for
	$\mu$-a.s. } x\in X
\end{equation} for every  $f\in\mathcal{L}^1(X,\mathcal{X},\mu)$. The product of $(X,\mathcal{X},\mu,T)$ with itself relative to  factor $(Y,\mathcal{Y},\nu,S)$ is the MDS
$$(X\times X,\mathcal{X}\times \mathcal{X},\mu \times_Y\mu, T\times T),$$ where the measure $\mu\times_Y\mu=\int_{Y}(\mu_y\times \mu_y) \, d \nu(y),$ which is   $T\times T$-invariant and  is supported on $R_\pi:=\{(x_1,x_2)\in X\times
X:\pi(x_1)=\pi(x_2)\}.$ 

The disintegration of measures has another equivalent form of expression. Specifically, denoting $\mathcal{S}:=\pi^{-1}\mathcal{Y}$, $\{\mu_x^{\mathcal{S}}\}_{x\in X}$ is the disintegration of $\mu$ relative to $\mathcal{S}$  if the followings hold:
\begin{itemize}
	\item $\mu_x^{\mathcal{S}}$ is a probability measure on $(X,\mathcal{X})$ with $\mu_x^{\mathcal{S}}(\pi^{-1}\pi(x))=1$ for $\mu$-a.s. $x\in X$;
	\item for $\mu$-a.s. $x\in X$, one has that for any $x_1,x_2\in\pi^{-1}\pi(x)$, $\mu_{x_1}^{\mathcal{S}}=\mu_{x_2}^{\mathcal{S}}$;
	\item for each $f \in \mathcal{L}^1(X,\mathcal{X},\mu)$,  one  has that $f \in \mathcal{L}^1(X,\mathcal{S}, \mu_x^{\mathcal{S}})$ for $\mu$-a.s. $x\in X$, the map $x \mapsto
	\int_X f\,d\mu_x^{\mathcal{S}}$ belongs to $\mathcal{L}^1(X,\mathcal{X}, \mu)$ and $\int_X \left(\int_X f\,d\mu_x^{\mathcal{S}} \right)\, d\mu(x)=\int_X f \,d\mu$.
\end{itemize}
We remark that the two forms above-mentioned are equivalent. Namely, $\mu_x^{\mathcal{S}}=\mu_{\pi(x)}$ for $\mu$-a.s. $x\in X$.

	\subsection{Relative Pinsker $\sigma$-algebra}In this
subsection, we introduce some notations and results on the relative
Pinsker $\sigma$-algebra.

 Let $(X,\mathcal{X},\mu,T)$ be a MDS  on a standard probability space and $\mathcal{S}$ be  a $T$-invariant sub-$\sigma$-algebra  of $\mathcal{X}$. The \emph{relative Pinsker $\sigma$-algebra $\mathcal{P}_\mu(\mathcal{S})$}
 with respect to $\mathcal{S}$ is defined as the smallest $\sigma$-algebra containing
 $$\{ A\in\mathcal{X}:  h_\mu(T,\{A, A^c\}|\mathcal{S})=0\}.$$
  By \cite[Section 4.10]{Wal},  $\mathcal{P}_\mu(\mathcal{S})$ is a $T$-invariant sub-$\sigma$-algebra of $\mathcal{X}$. Hence,  it
uniquely (up to an isomorphism) determines a factor $(Y,\mathcal{Y},\nu,S)$ of $(X,\mathcal{X},\mu,T)$ (see \cite[Theorem 6.5]{EW}). That is, there
exists a factor map 
$$\pi:(X,\mathcal{X},\mu,T)\rightarrow (Y,\mathcal{Y},\nu,S)$$ between two MDSs on standard probability spaces such that $\pi^{-1}(\mathcal{Y})=\mathcal{S}\pmod\mu$. Usually,  $\pi:(X,\mathcal{X},\mu,T)\rightarrow (Y,\mathcal{Y},\nu,S)$ is called \emph{relative Pinsker factor map}  with  respect to $\mathcal{S}$.
If $\mathcal{S}=\{X,\emptyset\}$, then $\mathcal{P}_\mu(T):=\mathcal{P}_\mu(\mathcal{S})$ is called  \emph{Pinsker $\sigma$-algebra} of  $(X,\mathcal{X},\mu,T)$.

Following lemma is  well-known  in  MDSs.  This lemma   will be used in the proof of Theorem \ref{22-03-06-05} and the main theorems.  The reader  can refer to  \cite[Theorem 2.1 and Theorem 2.3]{BGKM} or \cite[Lemma 4.1]{ZHANG}).
\begin{lem} \label{key-lem} 
Assume that $\pi:(X,\mathcal{X},\mu,T)\rightarrow (Z,\mathcal{Z},\eta,R)$ is  a  factor map between two  ergodic  MDSs on standard probability spaces. Let  $\pi_1:(X,\mathcal{X},\mu,T)\rightarrow
(Y,\mathcal{Y},\nu ,S)$ be the relative Pinsker factor map with  respect to
$\pi^{-1}\mathcal{Z}$ and  $\mu=\int_Y  \mu_y d \nu(y)$ be the disintegration of $\mu$ relative to the
factor $(Y,\mathcal{Y},\nu,S)$. If  $h_\mu(T|\pi)>0$, then $\mu_y$ is non-atomic (i.e. $\mu_y(\{x\})=0$ for each $x\in  X$) for $\nu$-a.s. $y\in Y$.
 \end{lem}

Finally, we give a proposition which describes  the relation between two different  relative Pinsker factors. The reader  can see a more general form in \cite[Theorem 0.4 (iii)]{MR1878075}.
\begin{prop}\label{22-03-06-03}
	Let  $\pi:(X,\mathcal{X},\mu,T)\rightarrow (Z,\mathcal{Z},\eta, R)$  be a factor map between two MDSs on standard probability spaces and $\pi_1: (X,\mathcal{X},\mu,T)\rightarrow (Y,\mathcal{Y},\nu, S)$ be the relative Pinsker factor map with respect to  $\pi^{-1}\mathcal{Z}$.  Denoting $\mu=\int_Y\mu_yd\nu(y)$ as the disintegration of $\mu$ relative to the factor $(Y,\mathcal{Y},\nu,S)$ and $$(X\times X,\mathcal{X}\times \mathcal{X},\mu \times_Y\mu, T\times T),$$ as the product of $(X,\mathcal{X},\mu,T)$ with itself relative to  factor $(Y,\mathcal{Y},\nu,S)$,  then one has that 
	\begin{align}
	\mathcal{P}_{\lambda}\big((\pi\circ\mathrm{Proj}_1)^{-1}\mathcal{Z}\big)=(\pi_1\circ\mathrm{Proj}_1)^{-1}\mathcal{Y}\pmod\lambda
	\end{align}
	where $\lambda=\mu\times_{Y}\mu$ and $\mathrm{Proj}_1: X\times X\to X$ is the projection to the first coordinate.
\end{prop}
\section{Two Fundamental Theorems}
\label{22-03-08-01}

Letting $X$ be a Borel subset of a Polish space $\tilde X$ and $\pi : (X,\mathcal{X},\mu, T)\rightarrow(Z,\mathcal{Z},\eta, R)$\footnote{In this place, $\mathcal{X}$ is the $\sigma$-algebra generated by the Borel subsets of $X$.} be a factor map between two MDSs on standard probability spaces, we prove that the stable sets are dense in the support of conditional measure with respect to the relative Pinsker factor,  namely  Theorem \ref{infinity}.  Additionally  assuming that $h_{\mu}(T|\pi)>0$ and $ (X,\mathcal{X},\mu, T)$ is ergodic, one has that the set of all separating pairs  has the full measure   under  the  product of 
conditional measure with respect to the relative Pinsker factor, namely Theorem \ref{22-03-06-05}.  

The reason why we don't directly  assume   that $(X,\mathcal{X},\mu, T)$ is a Polish system is that when we prove Theorem \ref{main1}, the random compact set may not be a Polish space but only a Borel subset of the whole space.

\subsection{Stable sets}
  The following theorem is a generalization  of \cite[Lemma 3.2]{WCZ}  to our setting. In the proof of it, the main difficulty is that the previous results depend on the compactness of the space.  Through some more detail observations, we can prove that, in non-compact systems, there also are many asymptotic pairs.
  \begin{thm}\label{infinity}
Let $X$ be a Borel subset of a Polish space $\tilde X$ and $\pi : (X,\mathcal{X},\mu, T)\rightarrow(Z,\mathcal{Z},\eta, R)$ be a factor map between two invertible MDSs on standard probability spaces.  Denote  $\pi_1 : (X,\mathcal{X},\mu, T)\rightarrow (Y,\mathcal{Y},\nu, S)$ as the Pinsker factor map  with respect to $\pi^{-1}\mathcal{Z}$ and $\mu=\int_Y\mu_yd\nu(y)$ as the disintegration of $\mu$ relative to  $(Y,\mathcal{Y},\nu, S)$.  For any infinite positive integer sequence $\ba=\{a_n\}_{n\in\mathbb{N}}$ $($i.e.  $\lim_{n\to+\infty}a_n=+\infty)$,  there exists a $\mu$-full measure subset $X_1$ such that  for any $x\in X_1$,
\begin{align}
\label{22-11-27-01}
\overline{W_{\ba}^s(x,T)\cap\supp(\mu_{\pi_1(x)})}=\supp(\mu_{\pi_1(x)})
\end{align}
where  $W_{\ba}^{s}(x,T)=\{y\in X: \lim_{n\to+\infty}d(T^{a_n}x, T^{a_n}y)=0\}$ and   $d$ is a compatible  complete metric on $\tilde X$.
\end{thm}

\begin{rem}
In  Theorem \ref{infinity}, for a probability measure $\hat{\mu}$ on $(X,\mathcal{X})$,  $\supp(\hat{\mu})$ is the smallest closed subset of $\tilde{X}$ with $\hat{\mu}(\supp(\hat{\mu}))=1$. The closure in \eqref{22-11-27-01} is with respect to the topology of $\tilde{X}$.
\end{rem}
In order to prove Theorem \ref{infinity}, let us begin with some lemmas.
\begin{lem}\label{PC}
Assume that $(X,\mathcal{X},\mu)$ is a Polish probability space. Then for any $A\in\mathcal{X}$ and $\epsilon>0$, there exists a compact subset $A_\epsilon\subset A$ such that $\mu(A\setminus A_\epsilon)\leq \epsilon$. 
\end{lem}
\begin{proof}
This lemma is a combining result from \cite[Theorem 1.3]{Billingsley} and the standard argument in measure theory.
\end{proof}

%\footnote{For convenience,   assume that an invertible MDS $(X,\mathcal{X},\mu, T)$ means that $T^{-1}$ exists on $X$ in this section. This assumption doesn't make effect on our result.}.   
 
  Let $(X,\mathcal{X}, \mu, T)$ be a  MDS. A partition $\xi$ is called a measurable partition of $(X,\mathcal{X})$ if 
  $$\xi=\bigvee_{i\in I}\xi_i$$
   where $\{\xi_i\}_{ i\in I}$ is a countable family of finite measurable partitions.  Denote $\sigma(\xi)$ as the smallest 
sub-$\sigma$-algebra of $\mathcal{X}$ which contains $\xi$.  For  two measurable partitions $\xi_1$ and $\xi_2$ of $(X,\mathcal{X})$,  it can be shown that $\sigma(\xi_1)\vee\sigma(\xi_2) = \sigma(\xi_1\vee\xi_2)$ where $\sigma(\xi_1)\vee\sigma(\xi_2)$ is   the smallest sub-$\sigma$-algebra of $\mathcal{X}$ that contains the $\sigma$-algebras
$\sigma(\xi_1)$ and $\sigma(\xi_2)$, so there is no ambiguity
to denote $\sigma(\xi_1\vee\xi_2)$ by $\xi_1\vee\xi_2$. For convenience, we write $\xi_1\preceq\xi_2$ if  for any element of $\xi_2$ is contained in some element of $\xi_1$. For a  measurable partition $\xi$ of  $(X,\mathcal{X},\mu,T)$, put $\xi(x)$ as the element of $\xi$ containing $x$,
$$\xi^-=\bigvee_{n\in\mathbb{N}}T^{-n}\xi\quad \text{and}\quad \xi^T=\bigvee_{n\in\mathbb{Z}}T^{-n}\xi.$$
The measurable partition $\xi$  is called \emph{measurable generating partition} if $\xi^T=\mathcal{X} \pmod\mu$.  
Now, we recall \cite[Lemma 3.1, Theorem 3.3, Lemma 3.5 and Lemma 3.6]{ZHANG} as follows.

\begin{lem}\label{H-1}
For an invertible MDS $(X,\mathcal{X}, \mu, T)$ on standard probability space,  let $\alpha$, $\beta$, $\gamma$ be finite measurable partitions of $(X,\mathcal{X})$ and $\mathcal{S}$ be a  $T$-invariant sub-$\sigma$-algebra of $\mathcal{X}$. 
Then, we have 
\begin{enumerate}[(i)]
\item\label{22-03-05-02} If $\alpha\preceq\beta$,  then $\lim_{n\rightarrow+\infty}H_{\mu}(\alpha|\beta^-\vee T^{-n}\gamma^{-}\vee\mathcal{S})=H_\mu(\alpha|\beta^-\vee\mathcal{S}).$
\item\label{22-03-05-03} $H_\mu (\alpha|\alpha^- \vee \mathcal{P}_{\mu}(\mathcal{S}))=H_{\mu}(\alpha|\alpha^{-} \vee\mathcal{S})=h_\mu(T,\alpha|\mathcal{S})$ and 
\[h_\mu(T,\alpha\vee\beta|\mathcal{S})=h_\mu(T,\beta|\mathcal{S})+h_\mu(T,\alpha|\beta^T\vee \mathcal{S}).\]

\item \label{22-03-05-04}If  $\xi$ is  a measurable generating partition of $(X,\mathcal{X})$ with $\xi \supset \mathcal{S}\pmod \mu$, then
$$\bigcap_{n\in\mathbb{N}_0}T^{-n}\xi^-\supset \mathcal{P}_{\mu}(\mathcal{S})\pmod\mu.$$
\end{enumerate}
\end{lem}

Following ideas in \cite[Lemma 3.7]{ZHANG}, we prove the corresponding result for our setting. This is a key lemma to prove Theorem \ref{infinity}.
\begin{lem}\label{partition}
Let $X$ be a Borel subset of a Polish space $\tilde X$ and $(X,\mathcal{X},\mu, T)$ be an invertible MDS  and  $\mathcal{S}$ be a $T$-invariant sub-$\sigma$-algebra of $\mathcal{X}$. Then $(X, \mathcal{X}, T, \mu)$ admits a measurable generating partition $\alpha$ 
of $(X,\mathcal{X})$ with following properties:
\begin{enumerate}[(a)]
\item $\alpha\supset\mathcal{S}\pmod\mu$ and $\mathcal{P}_\mu(\mathcal{S})=\bigcap_{n\in\mathbb{N}_0}T^{-n}\alpha^-\pmod\mu$, 
\item  the set
$$\{x\in  X: \text{any pair of points belonging to $\alpha^{-}(x)$ is asymptotic along $\mathbb{N}$\footnotemark[1]}\}$$
is $\mu$-full measure.
\footnotetext[1]{A  pair $(x,y)\in X\times X$ is called asymptotic along  $\ba=\{a_n\}_{n\in\mathbb{N}}$ if $\lim_{n\to+\infty}d(T^{a_n}x,T^{a_n}y)=0$, where $d$ is the compatible complete metric on $\tilde{X}$,}
\end{enumerate}
\end{lem}
\begin{proof}
  By Lemma \ref{PC}, we can find a  sequence of compact subsets  $\{X_n\}_{n\in\mathbb{N}}$ of $X$ with $\mu(X_n)\geq 1-\frac{1}{2^n}$ and $X_n\subset X_{n+1}$ for any $n\in\mathbb{N}$. 
  Therefore, we can choose a  finite measurable partition $\xi_n$  of $(X_n,\mathcal{X}_n)$   for each $n\in\mathbb{N}$ where $\mathcal{X}_n=\{A\cap X_n:  A\in\mathcal{X}\}$,  with the following properties:
 \begin{enumerate}[(I)]
 \item  $\xi_n\preceq\xi_{n+1}$ for each $n\in\mathbb{N}$;
 \item\label{22-03-05-06} the maximal diameter of $\xi_{n}$ goes to zero as $n\rightarrow +\infty$, namely, $\lim_{n\to+\infty}\max\{\text{diam}(A): A\in\xi_n\}=0$. 
 \end{enumerate}
 For each $n\in\mathbb{N}$, denoting $\mathcal{U}_{n}=\xi_{n}\cup\{X\setminus X_n\}$ which is  a  finite measurable partition of $(X,\mathcal{X})$, then $\mathcal{U}_{n+1}\succeq \mathcal{U}_n$ and $\bigvee_{n\in\mathbb{N}} \mathcal{U}_{n}=\mathcal{X}\pmod\mu$.  
 
 Define inductively  $\mathcal{V}_n=\mathcal{V}_{n-1}\vee T^{-t_{n-1}}\mathcal{U}_{n-1}$ for $n\geq 2$, where $\mathcal{V}_1=\mathcal{U}_1$ and $t_n\in\mathbb{N}$ are to be chosen by using \eqref{22-03-05-02}  of Lemma \ref{H-1} such that
$$H_\mu(\mathcal{V}_m|\mathcal{V}_{n-1}^-\vee\mathcal{S})-H_\mu(\mathcal{V}_m|\mathcal{V}_{n}^-\vee\mathcal{S})<\frac{1}{m}\frac{1}{2^{n-m}},\qquad m=1,2,\cdots,n-1.$$
Fixing $m\in\mathbb{N}$, for any $t>m$, one has that 
\begin{align}
\notag&H_\mu(\mathcal{V}_m|\mathcal{V}_{m}^-\vee\mathcal{S})-H_\mu(\mathcal{V}_m|\mathcal{V}_{t}^-\vee\mathcal{S})
\\\label{fine}=&\sum_{i=m}^{t-1} \left(H_\mu(\mathcal{V}_m|\mathcal{V}_{i}^-\vee\mathcal{S})-H_\mu(\mathcal{V}_m|\mathcal{V}_{i+1}^-\vee\mathcal{S})\re)\leq\frac{1}{m}.
 \end{align}
Clearly, $\beta=\bigvee_{n\in\mathbb{N}}\mathcal{V}_n$ is a measurable generating partition of $(X,\mathcal{X})$. Due to \eqref{fine}, one has that
\begin{equation}\label{lim}
\lim_{m\rightarrow+\infty} \lk(H_\mu(\mathcal{V}_m|\mathcal{V}_{m}^-\vee\mathcal{S})-H_\mu(\mathcal{V}_m|\beta^-\vee\mathcal{S})\re)= 0.
\end{equation}
\begin{claim}
\label{22-03-26-01}
$\bigcap_{n\in\mathbb{N}_0}(T^{-n}\beta^-\vee\mathcal{S})\subset\mathcal{P}_\mu(\mathcal{S})\pmod\mu$.
\end{claim}
\begin{proof}
We only need to prove that for any a finite measurable partition $\xi$ of $(X,\mathcal{X})$ with $\sigma(\xi)\subset\bigcap_{n\in\mathbb{N}_0}(T^{-n}\beta^-\vee\mathcal{S})$,    one has that $\sigma(\xi)\subset \mathcal{P}_\mu(\mathcal{S})\pmod\mu$. By \eqref{22-03-05-03} of Lemma \ref{H-1} and $\xi^{T}\subset \beta^-\vee\mathcal{S}$, one has
\begin{align*}
H_{\mu}(\xi|\xi^{-}\vee\mathcal{S})=& H_{\mu}(\mathcal{V}_m\vee\xi|\mathcal{V}_m^-\vee\xi^{-}\vee\mathcal{S})-H_{\mu}(\mathcal{V}_m|\mathcal{V}_m^-\vee\xi^{T}\vee\mathcal{S})\\
\leq&H_{\mu}(\xi|\mathcal{V}_m^-\vee\mathcal{S})+H_{\mu}(\mathcal{V}_m|\mathcal{V}_m^-\vee\mathcal{S})-H_{\mu}(\mathcal{V}_m| \beta^-\vee\mathcal{S}).
\end{align*}
As $m\rightarrow+ \infty$, by (\ref{lim}) we have
$$h_{\mu}(T,\xi|\mathcal{S})=H_{\mu}(\xi|\xi^-\vee\mathcal{S})\leq H_\mu(\xi|\beta^-\vee\mathcal{S})=0,$$
 which implies the claim.
\end{proof}
Take a measurable partition $\gamma$ of $(X, \mathcal{X})$ such that $\sigma(\gamma)=\mathcal{S}\pmod\mu$. Then $\alpha=\beta\vee\gamma$ is the required measurable partition.   Indeed,  $\alpha$ is a measurable generating partition with $\sigma(\alpha)\supseteq \mathcal S\pmod\mu$. By  \eqref{22-03-05-04} of Lemma \ref{H-1},  we have $\mathcal{P}_{\mu}(\mathcal S)\subset\bigcap_{n\in\mathbb{N}_0}T^{-n}\alpha^-\pmod\mu$. On the other hand, by  Claim \ref{22-03-26-01}, one has that
$$\bigcap_{n\in\mathbb{N}_0}T^{-n}\alpha^-=\bigcap_{n\in\mathbb{N}_0}(T^{-n}\beta^-\vee T^{-n}\gamma^{-})\subset\bigcap_{n\in\mathbb{N}_0}(T^{-n}\beta^-\vee\mathcal{S})\subset \mathcal{P}_\mu(\mathcal{S})\pmod\mu.$$
Therefore, $\mathcal{P}_{\mu}(\mathcal S)=\bigcap_{n\in\mathbb{N}_0}T^{-n}\alpha^-\pmod\mu$. Moreover, we have the following claim. 
\begin{claim}\label{2}
If  $x\notin\bigcup_{m\in\mathbb{N}}\bigcap_{j\in\mathbb{N}}\bigcup_{i\geq j}T^{-m-t_k}(X\setminus X_i)$,     then for any pair of    points  belonging to $\alpha^{-}(x)$  is asymptotic along $\mathbb{N}$. 
\end{claim}
\begin{proof}
Since  
$$\bigvee_{m\in\mathbb{N}}T^{-m}\alpha\succeq\bigvee_{m\in\mathbb{N}}\bigvee_{j\in\mathbb{N}}T^{-m}\mathcal{V}_j\succeq \bigvee_{m\in\mathbb{N}}\bigvee_{j\in\mathbb{N}}T^{-m-t_{j}}\mathcal{U}_{j},$$
 there exist $A_n\in\mathcal{U}_n$ for each $n\in\mathbb{N}$ such that $x\in \alpha^{-}(x)\subset\bigcap_{m\in\mathbb{N}}\bigcap_{j\in\mathbb{N}}T^{-m-t_j}A_j$. By assumption that $x\notin\bigcup_{m\in\mathbb{N}}\bigcap_{j\in\mathbb{N}}\bigcup_{i\geq j}T^{-m-t_i}(X\setminus X_i)$,  there  exists a strictly monotone increasing  positive integer sequence $\{k_n\}_{n=1}^\infty$ such that  $A_{k_n}\in\xi_{k_n}$ for each $n\in\mathbb{N}$.   Due to \eqref{22-03-05-06},  Claim \ref{2} holds.
\end{proof}
 Finally, Combining Claim \ref{2} and the fact that
 \begin{align*}
  \mu\left(\bigcup_{m\in\mathbb{N}}\bigcap_{j\in\mathbb{N}}\bigcup_{i\geq j}T^{-m-t_i}(X\setminus X_i)\right)&\leq \sum_{m\in\mathbb{N}} \mu\left(\bigcap_{j\in\mathbb{N}}\bigcup_{i\geq j}T^{-m-t_i}X_i^c\right)
  \\&=\sum_{m\in\mathbb{N}} \lim_{j\to+\infty}\mu\left(\bigcup_{i\geq j}T^{-m-t_i}X_i^c\right)
  \\&\leq \sum_{m\in\mathbb{N}} \lim_{j\to+\infty}\sum_{i\geq j}\frac{1}{2^i}=0.
  \end{align*}
We finish the proof of Lemma \ref{partition}.
  \end{proof}

Following result  is a corollary of  \cite[Corollary 5.21]{EW} and \cite[Theorem 3.17]{C}. For completeness, we present a  direct proof here.
\begin{lem}
\label{22-03-05-08}
Let $(X,\mathcal{X},\mu)$ be  a Polish probability space, $\{\mathcal{S}_n\}_{n\in\mathbb{N}}$ be a decreasing sub-$\sigma$-algebra of $\mathcal{X}$ and $\mathcal{S}=\bigcap_{n\in\mathbb{N}}\mathcal{S}_n$. Denote $\mu=\int_X\mu_{n,x}d\mu(x)$ and $\mu=\int_X\mu_{\infty,x}d\mu(x)$ as the disintegrations  relative to $\mathcal{S}_n$ and $\mathcal{S}$, respectively. Then for $\mu$-a.s. $x\in X$, the following inequality 
\begin{align}
\liminf_{n\to+\infty}\mu_{n,x}(U)\geq\mu_{\infty,x}(U)
\end{align}
 holds for any open subset $U$ of $X$.
\end{lem}
\begin{proof}
 Let $\{A_n\}_{n\in\mathbb{N}}$ be a countable topological basic of $X$.  By  Theorem \ref{M}, there exists $\mu$-full measure subset $X_1$ such that for any $x\in X_1$, we have 
\begin{align}
\lim_{n\to+\infty}\mu_{n,x}\left(\bigcup_{i\in I}A_i\right)=\mu_{\infty,x}\left(\bigcup_{i\in I}A_i\right)
\end{align}
for any finite subset $I\subset\mathbb{N}$. For any $x\in X_1$ and open subset $U:=\bigcup_{i\in\mathbb{N}}A_{k_i}$, one has 
\begin{align*}
\liminf_{n\to+\infty}\mu_{n,x}(U)&=\liminf_{n\to+\infty}\mu_{n,x}\left(\bigcup_{i\in\mathbb{N}}A_{k_i}\right)
\\&\geq\liminf_{m\to+\infty}\liminf_{n\to+\infty}\mu_{n,x}\left(\bigcup_{i=1}^{m}A_{k_i}\right)
\\&=\liminf_{m\to+\infty}\mu_{\infty,x}\left(\bigcup_{i=1}^{m}A_{k_i}\right)
\\&=\mu_{\infty,x}(U).
\end{align*}
This finishes the proof of Lemma \ref{22-03-05-08}.
\end{proof}
If $X$ is a compact metric space,   the following result   appeared in \cite[Lemma 2.1]{HLY}. By the analogous argument as \cite[Lemma 2.1]{HLY}, it still holds for our setting. We omit the proof here.
\begin{lem}\label{22-06-25-01}
Let $(X,\mathcal{X},\mu)$ a Polish probability space and $\mathcal{S}_2\subset \mathcal{S}_1$ be  two sub-$\sigma$-algebras of $\mathcal{X}$. Denote 
$\mu=\int_X\mu_{i,x}d\mu(x)$  as  the disintegration of $\mu$ relative to  $\mathcal{S}_i$ for $i=1,2$, respectively. Then $\supp(\mu_{1,x})\subset \supp(\mu_{2,x})$ for $\mu$-a.s. $x\in X$. 
\end{lem}

With help of lemmas above, we give the proof of  Theorem \ref{infinity} as follows.
\begin{proof}[Proof of Theorem \ref{infinity}]
 Recall that $X$ is a Borel subset of a Polish space $\tilde{X}$, $\pi : (X,\mathcal{X},\mu, T)\rightarrow(Z,\mathcal{Z},\eta, R)$ is a factor map between two MDSs on standard probability spaces, $\pi_1 : (X,\mathcal{X},\mu, T)\rightarrow (Y,\mathcal{Y},\nu, S)$ is the Pinsker factor  map  with respect to $\pi^{-1}\mathcal{Z}$ and $\mu=\int_Y\mu_yd\nu(y)$ be the disintegration of $\mu$ relative to the factor $(Y,\mathcal{Y},\nu, S)$. 

Fix an infinite positive integer sequence $\ba=\{a_n\}_{n\in\mathbb{N}}$.	
By Lemma \ref{partition}, there exists a  measurable partition $\alpha$ such that 
$$\mathcal{P}_{\mu}(\pi^{-1}\mathcal{Z})=\bigcap_{n\in\mathbb{N}_0}T^{-n}\alpha^-\pmod\mu$$ and  for $\mu$-a.s. $x\in X$, any pair of points belonging to  $\alpha^-(x)$  is asymptotic along $\mathbb{N}$. Hence $(T^{-n}\alpha^-)(x)\subset W_{\ba}^s(x,T)$ for each $n\in \mathbb{N}_0$ and $\mu$-a.s. $x\in X$. For any $n\in\mathbb{N}_0$,  let $\mu=\int_X\mu_{n,x}d\mu(x)$ be the disintegration of $\mu$ relative to $T^{-n}\alpha^-$. Following from the definition of the measure disintegration,  for any $n\in\mathbb{N}_0$ and $\mu$-a.s. $x\in X$, one has that  
\begin{equation}\label{4}
\mu_{n,x}(W_{\ba}^s(x,T))=1.
\end{equation}
Note that
$$\alpha^-\supset T^{-1}\alpha^-\supset T^{-2}\alpha^-\supset\cdots \text{ and } \bigcap_{n\in\mathbb{N}_0}T^{-n}\alpha^-=\mathcal{P}_{\mu}(\pi^{-1}\mathcal{Z})\pmod\mu.$$ 
Applying   Lemma \ref{22-03-05-08} and  Lemma \ref{22-06-25-01} on $\{T^{-n}\alpha^-\}_{n\in\mathbb{N}_0}$ and $\mathcal{P}_{\mu}(\pi^{-1}\mathcal{Z})$ for the Polish probability space $(\tilde{X},\tilde{\mathcal{X}},\mu)$ where $\tilde{\mathcal{X}}$ is the Borel-$\sigma$ algebra of $\tilde{X}$,  there exists a $\mu$-full measure subset  $X'\subset X$  such that for any $x\in X'$, any closed subset $F$ of $\tilde{X}$ and any $m\in\mathbb{N}$,  one has that
 \begin{equation}\label{3}
\limsup_{n\to+\infty}\mu_{n,x}(F)\leq\mu_{\pi_1(x)}(F),
\end{equation}
and 
\begin{equation}\label{supp}
\supp(\mu_{m,x})\subseteq \supp(\mu_{m+1,x})\subseteq\cdots\subseteq \supp(\mu_{\pi_1(x)}).
\end{equation}
  By \eqref{4} and \eqref{supp}, one has that for $x\in X'$ and any $n\in\mathbb{N}_0$,
$$\mu_{n,x}(\overline{W_{\ba}^s(x,T)\cap \supp(\mu_{\pi_1(x)})})\geq\mu_{n,x}(W_{\ba}^s(x,T)\cap \supp(\mu_{\pi_1(x)}))=1,$$ 
which implies that
$$\mu_{\pi_1(x)}(\overline{W_{\ba}^s(x,T)\cap \supp(\mu_{\pi_1(x)})})\overset{\eqref{3}}\geq\limsup_{n\to+\infty}\mu_{n,x}(\overline{W_{\ba}^s(x,T)\cap \supp(\mu_{\pi_1(x)})})=1.$$ Therefore,  $\supp(\mu_{\pi_1(x)})\subseteq \overline{W_{\ba}^s(x,T)\cap \supp(\mu_{\pi_1(x)})}$.

On the other hand, as  for any $x\in X'$,  it is clear that $\supp(\mu_{\pi_1(x)})\supseteq \overline{W_{\ba}^s(x,T)\cap\supp(\mu_{\pi_1(x)})}$. This completes the proof of Theorem \ref{infinity}.
\end{proof}

\subsection{Separating pairs}
In this subsection,   borrowing the  ideas in \cite{liu2022pinsker}, we  prove Theorem \ref{22-03-06-05}. In the proof of  \cite[Theorem 1.4]{liu2022pinsker}, the similar result has been  established for non-relative   factor case in compact systems. 
\begin{thm}\label{22-03-06-05}
	Let $X$ be a Borel subset of a Polish space $\tilde{X}$ and $\pi:(X,\mathcal{X},\mu,T)\rightarrow (Z,\mathcal{Z},\eta, R)$  be a factor map between two invertible ergodic MDSs on standard probability spaces.  Denote  $\pi_1: (X,\mathcal{X},\mu,T)\rightarrow (Y,\mathcal{Y},\nu, S)$ as the relative Pinsker factor map with respect to $\pi^{-1}\mathcal{Z}$, $\mu=\int_Y\mu_yd\nu(y)$ as the disintegration of $\mu$ relative to  the factor $ (Y,\mathcal{Y},\nu, S)$.  If $h_{\mu}(T|\pi)>0$, then for  any infinite  positive integer sequence $\ba=\{a_i\}_{i\in\mathbb{N}}$ $($i.e. $\lim_{i\to+\infty}a_i=+\infty)$,  there exists a $\nu$-full measure set $Y_0$ such that for any $y\in Y_0$, there exists  a  positive constant $\delta_y$ such that
	\begin{align*}
	\mu_y\times\mu_y(\mathcal{W}_{\ba}(X,\delta_y))=1,
	\end{align*}
	where 
	\begin{align*}
	\label{22-05-16-01}
	\mathcal{W}_{\ba}(X,\delta_y)=\Big\{(x_1, x_2)\in X\times X: \limsup_{n\to\infty}\frac{1}{n}\sum_{i=1}^nd(T^{a_i}x_1, T^{a_i}x_2)>\delta_y\Big\},
	\end{align*}
	$d$ is a compatible  complete metric on $\tilde X$.
\end{thm}

To prove Theorem \ref{22-03-06-05}, we recall the notion of the characteristic $\sigma$-algebra firstly, and then prove that relative Pinsker algebra is the characteristic $\sigma$-algebra along any infinite positive integer sequence  for MDSs, namely Proposition \ref{21-03-06-02}. 
\begin{defn}
		Let  $(X,\mathcal{X},\mu,T)$ be a MDS. A sub-$\sigma$-algebra $\mathcal{S}$ of $\mathcal{X}$ is a characteristic $\sigma$-algebra  for the positive integer sequence $\ba=\{a_n\}_{n\in\mathbb{N}}$ if for any $f\in\mathcal{L}^2(X,\mathcal{X},\mu)$,
		\begin{align}
		\lim_{N\to+\infty}\|\textbf{A}(\ba_N, f)-\mathbf{A}(\ba_N,  \mathbb{E}(f|\mathcal{S}))\|_{2}=0
		\end{align}
		where $\textbf{A}(\ba_N, f)=\frac{1}{N}\sum_{i=1}^{N}f\circ T^{a_i}$ and $\|f\|_2:=\big(\int_X|f|^2d\mu\big)^{1/2}$.
	\end{defn}

In \cite{liu2022pinsker}, authors proved that Pinsker algebra is the characteristic $\sigma$-algebra along any infinite sequence in countable discrete amenable groups in compact metric spaces. By the proof, we can see that the result also holds for the non-compact spaces. We restate a convenient version of this theorem as follows.
	\begin{lem}
	\label{lem1}
	Let  $(X,\mathcal{X},\mu,T)$  be an invertible MDS on standard probability space. If $\ba=\{a_n\}_{n\in\mathbb{N}}$ is an infinite positive integer sequence,  then the Pinsker algebra of $(X,\mathcal{X},\mu,T)$ is a characteristic $\sigma$-algebra for $\ba$.
\end{lem}	
The corresponding result of Lemma \ref{lem1} also holds for relative Pinsker $\sigma$-algebra. Specifically,
	\begin{prop}
		\label{21-03-06-02}
		Let  $\pi:(X,\mathcal{X},\mu,T)\rightarrow (Z,\mathcal{Z},\eta, R)$  be an invertible factor map between two MDSs on standard probability space. If $\ba=\{a_n\}_{n\in\mathbb{N}}$ is an infinite positive integer sequence,  then $\mathcal{P}_{\mu}(\pi^{-1}\mathcal{Z})$ is a characteristic $\sigma$-algebra for $\ba$.
	\end{prop}
	\begin{proof}
	Let $\mathcal{P}_{\mu}(T)$ be the Pinsker $\sigma$-algebra of $(X,\mathcal{X},\mu,T)$. Then 
		\[\mathcal{P}_{\mu}(T)\subset\mathcal{P}_{\mu}(\pi^{-1}\mathcal{Z}),\]
	which implies that for any $f,g\in\mathcal{L}^2(X,\mathcal{X},\mu)$,
$$\left\langle f-\mathbb{E}\big(f|\mathcal{P}_{\mu}(\pi^{-1}\mathcal{Z})\big), \mathbb{E}\big(g|\mathcal{P}_{\mu}(\pi^{-1}\mathcal{Z})\big)-\mathbb{E}\big(g|\mathcal{P}_{\mu}(T)\big)\right\rangle=0.$$
Hence, for any $N\in\mathbb{N}$, 
\begin{align*}
&\Big(\|\textbf{A}(\ba_N, f)-\textbf{A}\big(\ba_N,  \mathbb{E}(f|\mathcal{P}_{\mu}(T))\big)\|_{2}\Big)^2
\\=&\Big(\|\textbf{A}(\ba_N, f)-\textbf{A}\big(\ba_N,  \mathbb{E}(f|\mathcal{P}_{\mu}(\pi^{-1}\mathcal{Z}))\big)\|_{2}\Big)^2
\\\quad&+\Big(\|\textbf{A}(\ba_N,  \mathbb{E}(f|\mathcal{P}_{\mu}(\pi^{-1}\mathcal{Z})))-\textbf{A}\big(\ba_N, \mathbb{E}(f|\mathcal{P}_{\mu}(T))\big)\|_{2}\Big)^2.
\end{align*}
That is, \[\|\textbf{A}(\ba_N, f)-\textbf{A}\big(\ba_N,  \mathbb{E}(f|\mathcal{P}_{\mu}(\pi^{-1}\mathcal{Z}))\big)\|_{2}\le
\|\textbf{A}(\ba_N, f)-\textbf{A}\big(\ba_N, \mathbb{E}(f|\mathcal{P}_{\mu}(T))\big)\|_{2}.\]
By Lemma \ref{lem1}, one has that\[\lim_{N\to+\infty}\|\textbf{A}(\ba_N, f)-\textbf{A}\big(\ba_N,  \mathbb{E}(f|\mathcal{P}_{\mu}(\pi^{-1}\mathcal{Z}))\big)\|_{2}=0.\]
We finish the proof of Proposition \ref{21-03-06-02}.
	\end{proof}

To prove Theorem \ref{22-03-06-05}, we recall the notion of the weak convergence and some results about it.
Let $(X,\mathcal{X},\mu)$ be a probability space. A sequence $\{h_n\}_{n\in\mathbb{N}}$ in $\mathcal{L}^{2}(X,\mathcal{X},\mu)$ converges weakly to $h\in\mathcal{L}^{2}(X,\mathcal{X},\mu)$ (denoted by $h_n\overset{w}\to h$),
if $\lim_{n\to+\infty}\int  h_nfd\mu=\int hf d\mu$ for all $f \in\mathcal{L}^{2}(X,\mathcal{X},\mu)$.  The following results  are proved in \cite[Lemma 4.1 and Lemma 4.2]{liu2022pinsker}.  
\begin{lem}
\label{21-03-06-01}
Let  $(X,\mathcal{X},\mu,T)$  be a MDS and $\ba=\{a_n\}_{n\in\mathbb{N}}$ be  an infinite  positive integer sequence.  Given  a $T$-invariant sub-$\sigma$-algebra $\mathcal{S}$ of $\mathcal{X}$ and  $f\in\mathcal{L}^2(X,\mathcal{X},\mu)$, assume that there exists $ P(f)\in\mathcal{L}^2(X,\mathcal{X},\mu)$ such that $\textbf{A}(\ba_n, f)\overset{w}\rightarrow P(f)$. Then, 
\begin{enumerate}[(i)]
\item\label{21-03-06-01-01} if $\lim_{n\to+\infty}\| \textbf{A}(\ba_n, f)- \textbf{A}(\ba_n, \mathbb{E}(f|\mathcal{S}))\|_2=0,$  one has that
\begin{align*}
\textbf{A}(\ba_n, \mathbb{E}(f|\mathcal{S}))\overset{w}\rightarrow P(f),
\end{align*}
 as $n\to\infty$. Moreover, $ P(f)\in\mathcal{L}^2(X,\mathcal{S},\mu)$.
\item\label{21-03-06-01-02} if $f$ is real-valued, then $\limsup_{n\to+\infty}\textbf{A}(\ba_n, f)(x)\geq P(f)(x)$ for $\mu$-a.s. $x\in X$.
\end{enumerate}
\end{lem}
\begin{lem}\label{22-03-26-02}
Let  $(X,\mathcal{X},\mu,T)$  be a MDS and $\ba=\{a_n\}_{n\in\mathbb{N}}$ be  an infinite positive integer sequence.  Given   $g\in\mathcal{L}^2(X,\mathcal{X},\mu)$ with $g(x)>0$ for $\mu$-a.s. $x\in X$ and $A\in\mathcal{X}$ with $\mu(A)>0$,  assume that  $f=g1_A$ where $1_A$ is the character function of $A$,  and  $\textbf{A}(\ba_n, f)\overset{w}\rightarrow  P(f)\in\mathcal{L}^2(X,\mathcal{X},\mu)$. 
Then $P(f)(x)>0$ for $\mu$-a.s. $x\in A$. 
\end{lem}
\begin{proof}
Assuming that above lemma doesn't hold,  then  $\mu(B)>0$ where $B:=\{x\in A : P(f)(x)\leq 0\}$.  Since $f(x)>0$ for $\mu$-a.s. $x\in A$,  there exists $\epsilon>0$ such that $\mu(C_{\epsilon})<\mu(B)$, where $C_{\epsilon}=\{x\in A: f(x)\leq\epsilon\}$. 
For any $n\in\mathbb{N}$, 
\begin{align*}
\int_X(f\circ T^{n}) 1_Bd\mu=\int_Xf 1_{T^{-n}B}d\mu\geq \int_{ T^{-n}B\setminus C_{\epsilon}}fd\mu\geq\epsilon\big(\mu(B)-\mu(C_{\epsilon})\big),
\end{align*}
which implies that 
\begin{align*}
\int_XP(f) 1_Bd\mu=\lim_{n\to+\infty}\int_X \textbf{A}(\ba_n, f)1_Bd\mu\geq\epsilon\big(\mu(B)-\mu(C_{\epsilon})\big)>0
\end{align*}
which contradicts the fact that $\int_XP(f)1_Bd\mu\leq0$.
\end{proof}
\begin{lem}\label{factor}
Assume that $\pi:(X,\mathcal{X},\mu,T)\rightarrow (Y,\mathcal{Y},\nu,S)$  is a factor map between two 
MDSs on standard probability spaces and $A\in\mathcal{X}$ with $\mu(A)>0$.  Denoting $\mu=\int_Y \mu_yd\nu(y)$ as the disintegration of $\mu$ relative to the factor $(Y,\mathcal{Y},\nu,S)$, then there is a subset $B\in\mathcal{Y}$ with $\nu(B)\geq \mu(A)$  such that any $y\in B$, $\mu_y(\pi^{-1}(y)\cap A)>0$.
\end{lem}
\begin{proof}
Let $B=\{y\in Y: \mu_y(A)>0\}$. Then $B\in\mathcal{Y}$  and 
$$\nu(B)=\int_B 1d\nu(y)\geq\int_B \mu_y(A)d\nu(y)=\int_Y \mu_y(A)d\nu(y)=\mu(A).$$
By the definition of measure disintegration,  it is clear that for any $y\in B$, $\mu_y(\pi^{-1}(y)\cap A)>0$. 
\end{proof}

After all preparations, we begin to  prove Theorem \ref{22-03-06-05}.
\begin{proof}[Proof of Theorem \ref{22-03-06-05}]
Recall that  $X$ is  a Borel subset of a Polish space $\tilde{X}$, $\pi:(X,\mathcal{X},\mu,T)\rightarrow (Z,\mathcal{Z},\eta, R)$  is a factor map between two 
MDSs on standard probability spaces,  $\pi_1: (X,\mathcal{X},\mu,T)\rightarrow (Y,\mathcal{Y},\nu, S)$ is the relative Pinsker factor map with respect to $\pi^{-1}\mathcal{Z}$ and  $d$ is a compatible complete metric on $\tilde{X}$. Let $\mu=\int_Y\mu_yd\nu(y)$ be the disintegration of $\mu$ relative to $(Y,\mathcal{Y},\nu, S)$.  
\begin{claim}
\label{22-05-13-01}
Given a compact subset $X_1$ of $X$ with $\mu(X_1)>0$ and an infinite positive integer sequence $\ba=\{a_i\}_{i\in\mathbb{N}}$,  there exists a measurable set $Y_1$ with $\nu(Y_1)\geq\mu\times_{Y}\mu(X_1\times X_1)$ such that for any $y\in Y_1$, there exists $\delta_y(1)>0$ such that
	\begin{align*}
	\mu_y\times\mu_y\big(\mathcal{W}_{\ba}(X_1,\delta_y(1))\big)=\mu_y\times \mu_y(X_1\times X_1)>0,
	\end{align*}
	where 
		\begin{align*}
		\mathcal{W}_{\ba}(X_1,\delta_y(1))=\Big\{(x_1, x_2)\in X_1\times X_1: \limsup_{n\to+\infty}\frac{1}{n}\sum_{i=1}^nf_1(T^{a_i}x_1, T^{a_i}x_2)>\delta_y(1)\Big\},
		\end{align*}
		and  $f_1=1_{X_1\times X_1}d$.
\end{claim}
\begin{proof}
 Let $\lambda:=\mu\times_{Y}\mu$ (see  Section \ref{22-07-06-01}) and $\Delta_{X}=\{(x,x): x\in X \}$.  By Lemma \ref{key-lem} and Fubini's theorem,  we have
\begin{align*}
\lambda(\Delta_{X})=\int_Y\int_{X}\mu_y(x)d\mu_y(x)d\nu(y)=0,
\end{align*}
which implies that $d(x_1,x_2)>0$ for $\lambda$-a.s. $(x_1, x_2)\in X\times X$  and hence  $f_1(x_1,x_2)>0$ for any $\lambda$-a.s. $(x_1,x_2)\in X_1\times X_1$. 

 Denoting $\textbf{A}(\textbf{a}_N, f_1):=\frac{1}{N}\sum_{i=1}^Nf_1\circ(T\times T)^{a_i}$ for any $N\in\mathbb{N}$, it is clear that $\{\textbf{A}(\textbf{a}_N, f_1)\}_{N\in\mathbb{N}}$ is a sequence of bounded functions in $\mathcal{L}^2(X\times X, \mathcal{X}\times\mathcal{X},\lambda)$. By using Alaoglu's theorem,  there exists a monotone increasing positive integer sequence $\{ a_{k_N}\}_{N\in\mathbb{N}}$ and $P(f_1)\in \mathcal{L}^2(X\times X, \mathcal{X}\times\mathcal{X},\lambda)$ such that 
\begin{align}\label{8}
\textbf{A}(\mathbf{a}_{k_N}, f_1)\overset{w}\rightarrow P(f_1)\text{ as } N\to\infty.
\end{align}
Applying  Proposition \ref{21-03-06-02} to the  $\pi\circ\mathrm{Proj}_1: (X\times X,\mathcal{X}\times\mathcal{X},\lambda, T\times T)\to (Z,\mathcal{Z},\eta, R)$ where $\mathrm{Proj}_1: X\times X\to X$ is the projection to the first coordinate, we have 
$$\lim_{N\to+\infty}\|\textbf{A}(\ba_N, f_1)-\textbf{A}\big(\ba_N, \mathbb{E}(f_1|\mathcal{P}_{\lambda}((\pi\circ\mathrm{Proj}_1)^{-1}\mathcal{Z}))\big)\|_2=0.$$  Due to \eqref{8} and  Lemma \ref{21-03-06-01}, one has 
\begin{align*}
\textbf{A}\big(\mathbf{a}_{k_N}, \mathbb{E}(f_1|\mathcal{P}_{\lambda}((\pi_1\circ\mathrm{Proj}_1)^{-1}\mathcal{Z}))\big)\overset{w}\rightarrow P(f_1)\in \mathcal{L}^2(X\times X, \mathcal{P}_{\lambda}((\pi_1\circ\mathrm{Proj}_1)^{-1}\mathcal{Z}),\lambda),
\end{align*} 
as $N\to\infty$.  By using Proposition \ref{22-03-06-03},  one has that 
 \begin{align*}
 P(f_1)(x_1,x_2)&=\mathbb{E}\big(P(f_1)|\mathcal{P}_{\lambda}((\pi_1\circ\mathrm{Proj}_1)^{-1}\mathcal{Z})\big)(x_1, x_2)
 \\&=\mathbb{E}(P(f_1)|(\pi_1\circ\mathrm{Proj}_1)^{-1}\mathcal{Y})(x_1, x_2)=\int_{X\times X}P(f_1)d(\mu_y\times\mu_y)
 \end{align*}
 for $\lambda$-a.s. $(x_1, x_2)\in X\times X$ where $y=\pi_1\circ\mathrm{Proj}_1((x_1, x_2))$. 
Combining Lemma \ref{21-03-06-01}  and Lemma \ref{22-03-26-02},  there exists a measurable subset $B$ of $X_1\times X_1$ with $\lambda(B)=\lambda(X_1\times X_1)$ such that for any $(x_1, x_2)\in B$, we have 
$$\limsup_{N\to+\infty}\textbf{A}(\ba_N, f_1)(x_1,x_2)\geq P(f_1)(x_1, x_2)=\int_{X\times X}P(f_1)d(\mu_y\times\mu_y):=2\delta_y(1)>0,$$  where $y=\pi_1\circ\mathrm{Proj}_1((x_1, x_2))$.  

According to Lemma  \ref{factor}, there exists $Y'_1\in\mathcal{Y}$ with $\nu(Y'_1)\geq\lambda(B)=\lambda(X_1\times X_1)$ such that for any $y\in Y'_1$, 
$$\mu_y\times\mu_y\big((\pi_1\circ\mathrm{Proj}_1)^{-1}(y)\cap B\big)>0.$$
 As $\lambda((X_1\times X_1)\setminus B)=0$,  there exists $\nu$-full measure $Y''_2$, such that for any $y\in Y''_2$ 
 $$\mu_y\times\mu_y((X_1\times X_1)\setminus B)=0.$$  
Let $Y_1=Y'_1\cap Y''_2.$ Then $\nu(Y_1)\geq\lambda(X_1\times X_1)$. Since $(\pi_1\circ\mathrm{Proj}_1)^{-1}(y)\cap B \subset \mathcal{W}_{\ba}(X_1, \delta_y(1))$  for  any $y\in Y_1$,   it follows that for each $y\in Y_1$,
\begin{align*}
\mu_y\times \mu_y(\mathcal{W}_{\ba}(X_1,\delta_y(1)))&\geq\mu_y\times\mu_y\big((\pi_1\circ\mathrm{Proj}_1)^{-1}(y)\cap B\big)
\\&=\mu_y\times\mu_y(B)=\mu_y\times\mu_y(X_1\times X_1).
\end{align*}
This finishes the proof of Claim \ref{22-05-13-01}.
\end{proof}
According to Lemma \ref{PC},  there exists a  sequence of compact subsets $\{X_n\}_{n\in\mathbb{N}}$  of $X$  with $\mu(X_n)\geq 1-1/n$ and $X_n\subset X_{n+1}$ for each $n\in\mathbb{N}$. Applying  Claim \ref{22-05-13-01} on each $X_n$,  we can choose a subsequence  a monotone increasing positive integer sequence $\{ a_{k_N}\}_{N\in\mathbb{N}}$ such that for any $n\in\mathbb{N}$, 
$$\textbf{A}(\mathbf{a}_{k_N}, f_n)\overset{w}\rightarrow P(f_n)\text{ as } N\to\infty.$$
Furthermore, one has  that $\delta_y(n)\geq \delta_y(m)$  if $y\in Y_n\cap Y_m$ and $n\geq m$. Indeed,  we only need to prove that $P(f_n)\geq P(f_m)$ for any $n\geq m$.  For any $g\in \mathcal{L}^2(X,\mathcal{X},\mu)$ and $g\geq 0$, 
\begin{align*}
\langle P(f_n)-P(f_m), g\rangle&=\lim_{N\to+\infty}\langle\textbf{A}(\mathbf{a}_{k_N}, f_n)-\textbf{A}(\mathbf{a}_{k_N}, f_m), g\rangle\geq 0,
\end{align*}
since $f_n\ge f_m$.
Let $Y'=\bigcap_{i\in\mathbb{N}}\bigcup_{n\geq i}Y_n$.  Then, by Claim \ref{22-05-13-01} and the fact above, one has $\nu(Y')=1$. For each $y\in Y'$, assume that $y\in\bigcap _{n\in\mathbb{N}}Y_{l_n}$. Then, one has  
\begin{align*}
\mu_y\times\mu_y(\mathcal{W}_{\ba}(X,\delta_y))&=\liminf_{n\to+\infty}\mu_y\times\mu_y(\mathcal{W}_{\ba}(X_{l_n},\delta_y))
\\&\geq \liminf_{n\to+\infty}\mu_y\times\mu_y\Big(\mathcal{W}_{\ba}\big(X_{l_n},\delta_y(l_n)\big)\Big)
\\&=\liminf_{n\to+\infty}\mu_y\times\mu_y(X_{l_n}\times X_{l_n})=1,
\end{align*}
where $\delta_y:=\min\{\delta_y(n) : y\in Y_n\}$.  Therefore, Theorem \ref{22-03-06-05} holds.
 \end{proof}

\section{Proof of  the main Theorems}\label{22-03-08-02}
In this section, we  prove  Theorem \ref{main1} and Theorem \ref{main2} by  using Theorem \ref{infinity} and Theorem \ref{22-03-06-05}.  Recall a result (see \cite[Theorem 1]{M20220712}) which is crucial   to find  Cantor subsets in a perfect complete metric space, firstly.

\begin{lem} \label{Myc} 
Let $Y$ be a perfect complete metric space and $C$ be a dense
$G_\delta$ subset of $Y\times Y$. Then there exists a dense Mycielski subset $S\subseteq Y$ such that
$S\times S\subseteq C\cup \Delta_Y$, where $\Delta_Y=\{ (y,y):y\in Y\}$. 
\end{lem}

%%%%%%%%%%%%%%%%%%%%%%%%%%%%%%%%%%%%%%%%%%%%%%%%%%%%%%%%%%%%%%%%%%%%%%%%%%% \section{Main Theorem} In this
\subsection{Proof of Theorem \ref{main1}}\label{22-04-22-03}
In this subsection, we assume that  $(X,\phi)$ is an injective continuous random dynamical system over an invertible ergodic Polish system $(\Omega,\mathcal{F},\mathbb{P},\theta)$. In this case, for any $\phi$-invariant random compact set $K$ and for any $\mu\in\mathcal{M}_{\mathbb{P}}^{K}(\Omega\times X)$,  $(K,\mathcal{K},\mu,\Phi)$ is an invertible MDS on a standard probability space  (for example, see \cite[Theorem 2.8]{G}), where $\mathcal{K}$ is the $\sigma$-algebra generated by the Borel subset of $K$.
\begin{proof}[Proof of Theorem \ref{main1}]
 Assume that
$ h_{top}(\phi,K)>0$. By Proposition \ref{VP} there exists $\mu\in \mathcal{E}_{\mathbb{P}}^K(\Omega\times X)$  such
that $h_\mu(\phi,K)>0$. Since $\Phi$ is injective, 
$(K,\mathcal{K},\mu,\Phi)$ is an  invertible ergodic MDS on a standard probability space and $\pi_\Omega: (K,\mathcal{K},\mu,\Phi)\rightarrow (\Omega, \mathcal{F},\mathbb{P},\theta)$ is a factor map between two MDSs on standard probability spaces.  We divide the remainder of the proof into two steps.

\medskip \noindent{\bf{Step 1.}}  In this step, we obtain a $\nu$-full subset of $Y$, which  has some good properties to help us complete proof of Theorem \ref{main1}.

Let $\mathcal{P}_{\mu}(\pi_{\Omega}^{-1}\mathcal{F})$ be the relative Pinsker $\sigma$-algebra of
$\big(K, \mathcal{K},\mu,\Phi\big)$ with respect to $\pi_{\Omega}^{-1}\mathcal{F}$.  By \cite[Theorem 6.5 and Lemma 5.2]{EW}, there exists an invertible  Polish system  $(Y,\mathcal{Y},\nu,S)$ and  two factor maps
$$\pi_1: \big(K, \mathcal{K},\mu,\Phi\big)\rightarrow (Y,\mathcal{Y},\nu,S),\
\pi_2:(Y,\mathcal{Y},\nu,S)\rightarrow  (\Omega,\mathcal{F},\mathbb{P},\theta),$$ between invertible MDSs on standard probability space such that  $\pi_2\circ \pi_1=\pi_{\Omega}$  and
$\pi_1^{-1}(\mathcal{Y})=\mathcal{P}_{\mu}(\pi_{\Omega}^{-1}\mathcal{F})\pmod\mu$.   That is, $\pi_1: \big(K, \mathcal{K},\mu,\Phi\big)\to (Y,\mathcal{Y},\nu,S)$ is  the relative  Pinsker factor map
with respect   to $\pi_{\Omega}^{-1}\mathcal{F}$.

 Denote $ d_{\Omega}$ as the  complete metric on $\Omega$.  Define  the metric on $\Omega\times X$  as
 $$\rho(\boldsymbol{k}_1, \boldsymbol{k}_2)=\max\{d_{\Omega}(\omega_1,\omega_2), d(x_1, x_2)\}$$
   where $\boldsymbol{k}_i=(\omega_i, x_i)\in \Omega\times X$ for $i=1,2$. Then $(\Omega\times X, \rho)$ is a complete separable  metric space.

Given $\delta>0$ and $\boldsymbol{k}_0\in K$, put
\begin{equation*}\label{wm}
\mathcal{W}_{\ba}(K,\delta)=\{(\boldsymbol{k}_1, \boldsymbol{k}_2)\in K\times K: \limsup_{N\rightarrow +\infty}
 \frac{1}{N}\sum_{i=1}^{N} \rho\lk((\Phi\times\Phi)^{a_{i}}(\boldsymbol{k}_1, \boldsymbol{k}_2)\re)>\delta\},
\end{equation*} 
 and 
$$W_{\ba}^s(\boldsymbol{k}_0,\Phi)=\{\boldsymbol{k}\in K: \lim_{n\to+\infty}\rho(\Phi^{a_n}\boldsymbol{k},\Phi^{a_n}\boldsymbol{k}_0)=0\}.$$

 Let  $\mu=\int_Y \mu_y d \nu(y)$ be the disintegration relative to the  factor
$(Y,\mathcal{Y},\nu,S)$.      Claim that there is a $\nu$-full measure  set  $\hat{Y}$   such that for  $y\in \hat{Y}$ 
with following properties:
\begin{enumerate}[(P1)]
\item\label{22-07-07-01} Denoting $\omega=\pi_2(y)$,  then $\phi(n,\omega):X\rightarrow X $ is  continuous and $K(\theta^n\omega)$ is a compact subset of $X$  for each  $n\in \mathbb{N}$.
 \item\label{22-04-22-01-04} $\mu_y(\pi_2(y)\times K_y)=1$ and $\mu_y$ is non-atomic where $K_y:=\{x\in X: (\pi_2(y), x)\in K\}$.
 \item\label{22-04-22-01-05} There exists $\delta_y>0$ such that $\mu_y\times \mu_y(\mathcal{W}_{\ba}(K,\delta_y))=1.$
 \item\label{22-07-07-02}    For any $\boldsymbol{k}\in\pi_1^{-1}(y)$, one has that  $\overline{W^{s}_{\ba}(\boldsymbol{k}, \Phi)\cap \supp(\mu_{y})}=\supp(\mu_y).$
\end{enumerate}
(\hyperref[22-07-07-01]{P1}) is the basic assumption.   (\hyperref[22-04-22-01-04]{P2}) is due to $\mu(K)=1$ and Lemma \ref{key-lem}. Applying Theorem \ref{infinity} and  Theorem \ref{22-03-06-05} to the factor map  $\pi_{\Omega}: \big(K, \mathcal{K},\mu,\Phi\big)\to (\Omega,  \mathcal{F},\mathbb{P},\theta)$, respectively,   (\hyperref[22-04-22-01-05]{P3}) and (\hyperref[22-07-07-02]{P4}) hold.

\medskip \noindent{\bf{Step 2.}} 
In this step, we finish the proof of Theorem \ref{main1} by  showing the following lemma.
\begin{lem}
\label{22-03-06-08}
For any $y\in \hat{Y}$, there exists a Mycielski subset $S_y$ of  $K_y$ and  $\epsilon_0>0$ such that for any  two distinct points $x_1, x_2\in S_y$ satisfies  \eqref{22-03-09-01} and \eqref{22-03-09-02}. 
 \end{lem}

\begin{proof}
Given $y\in \hat Y$, let  
\begin{align*} 
E_y:=\{ x\in  K_y: \mu_y(
\pi_2(y) \times U)>0 \text{ for any open neighborhood } U \text{ of } x\}.  
\end{align*}

By (\hyperref[22-04-22-01-04]{P2}),  we have that  $E_y$  is a perfect  compact subsets and 
\begin{align}
\label{22-03-06-07}
\pi_2(y)\times E_y=\supp(\mu_y).
\end{align}
   Given  $\epsilon>0$, let $\omega(y):=\pi_2(y)\in \Omega$,  
\begin{align*}
\mathbf{P}(y)=\bigcap_{k=1}^{\infty} \bigcap_{l=1}^{\infty}\bigcup_{N=l}^{\infty}&\left\{(x_1,x_2)\in E_y\times E_y:  \frac{1}{N}\sum_{j=1}^N d\big(\phi(a_{j}, \omega(y))x_1, \phi(a_{j}, \omega(y))x_2\big)<\frac{1}{k}\right\},
\end{align*}
and 
\begin{align*}
\mathbf{D}_{\epsilon}(y)=\bigcap_{k=1}^{\infty}\bigcap_{l=1}^{\infty}\bigcup_{N=l}^{\infty}&\left\{(x_1,x_2)\in E_y\times E_y: \frac{1}{N}\sum_{j=1}^Nd\big(\phi(a_{j}, \omega(y))x_1, \phi(a_{j}, \omega(y))x_2\big)>\epsilon-\frac{1}{k}\right\}.
\end{align*}
We complete the proof of this lemma by the following two claims.
\begin{claim}
Recall that $\delta_y$ is defined as  (\hyperref[22-04-22-01-05]{P3}).  Then $\mathbf{D}_{\delta_y}(y)$ is a dense   subset of $(E_y\times E_y, d\times d)$.
\end{claim}
\begin{proof}
 For any non-empty open subsets $U_1$, $U_2$ of  $X$
with $(U_1\times U_2)\cap (E_y\times E_y)\neq \varnothing$, one has 
\begin{align*} &\mu_y\times
\mu_y\left(\Big(\big(\pi_2(y)\times (U_1\cap E_y)\big)\times \big(\pi_2(y)\times (U_2\cap E_y)\big)\Big)\cap \mathcal{W}_{\ba}(K,\delta_y)\right)
\\\overset{\rm(\hyperref[22-04-22-01-05]{P3})}=&\mu_y\times\mu_y\Big(\big(\pi_2(y)\times (U_1\cap E_y)\big)\times \big(\pi_2(y)\times (U_2\cap E_y)\big)\Big)
\\\overset{\eqref{22-03-06-07}}=&\mu_y\big(\pi_2(y)\times (U_1\cap E_y)\big)\cdot\mu_y\big(\pi_2(y)\times (U_2\cap E_y)\big)>0,
\end{align*}
which implies that  
$$\Big(\big(\pi_2(y)\times (U_1\cap E_y)\big)\times \big(\pi_2(y)\times (U_2\cap E_y)\big)\Big)\cap \mathcal{W}_{\ba}(K,\delta_y)\neq\varnothing.$$
Hence there exist $x_j\in U_j\cap E_y$ such that $(\boldsymbol{k}_1,\boldsymbol{k}_2)\in \mathcal{W}_{\ba}(K,\delta_y)$ where $\boldsymbol{k}_j=(\omega(y),x_j)$ for $j=1,2$. It follows that
\begin{align*}
&\limsup_{N\to+\infty}\frac{1}{N}\sum_{j=1}^N d\big(\phi(a_{j}, \omega(y))x_1, \phi(a_{j}, \omega(y))x_2\big)=\limsup_{N\to+\infty}\frac{1}{N}\sum_{j=1}^{N}\rho(\Phi^{a_j}(\boldsymbol{k}_1), \Phi^{a_j}(\boldsymbol{k}_2))>\delta_y.
\end{align*}
  Therefore, $\mathbf{D}_{\delta_y}(y)$ is dense in $ E_y \times E_y$. 
\end{proof}
\begin{claim}
$\mathbf{P}(y)$ is a   dense subset of $(E_y\times E_y,d\times d)$.
\end{claim}
\begin{proof}
By (\hyperref[22-07-07-02]{P4}), for any $\boldsymbol{k}\in \pi_1^{-1}(y)$ one has that
\begin{equation}\label{34}
\overline{W_{\ba}^s(\boldsymbol{k},\Phi)\cap \supp(\mu_y)}=\supp(\mu_y), 
\end{equation}
 Denoting 
\begin{align*}
\text{Asy}_{\ba}(K, \Phi):=\{(\boldsymbol{k}_1, \boldsymbol{k}_2)\in K\times K: \lim_{n\to\infty}\rho(\Phi^{a_n}(\boldsymbol{k}_1), \Phi^{a_n}(\boldsymbol{k}_2))=0\},
\end{align*} then $W_{\ba}^s(\boldsymbol{k}, \Phi)\times W_{\ba}^s(\boldsymbol{k}, \Phi)\subset \text{Asy}_{\ba}(K, \Phi)$. Hence
\begin{align*}
\text{Asy}_{\ba}(K, \Phi)\cap\big(\supp(\mu_y)\times\supp(\mu_y)\big)\supset &\big(W_{\ba}^s(\boldsymbol{k}, \Phi)\times W_{\ba}^s(\boldsymbol{k},\Phi)\big)\cap \big(\supp(\mu_y)\times  \supp(\mu_y)\big)
\\\supset&\big(W_{\ba}^s(\boldsymbol{k},\Phi)\cap \supp(\mu_y)\big)\times \big(W_{\ba}^s(\boldsymbol{k}, \Phi)\cap \supp(\mu_y)\big).
\end{align*}
This combined with \eqref{34} implies that 
\begin{align}
\label{22-07-07-05}
\overline{\text{Asy}_{\ba}(K,\Phi)\cap \big(\supp(\mu_y)\times\supp(\mu_y)\big)}=\supp(\mu_y)\times\supp(\mu_y).
\end{align}
Denoting $\pi_{X}: \Omega\times  X\to X$  as the projection, it is clear that 
$$\pi_X\times\pi_X\Big(\text{Asy}_{\ba}(K,\Phi)\cap\big((\pi_2(y)\times E_y)\times (\pi_2(y)\times E_y)\big)\Big)\subset \mathbf{P}(y).$$  It follows that 
\begin{align*}
 \overline{\mathbf{P}(y)}&\supset\pi_{X}\times \pi_{X}\Big(\big((\pi_2(y)\times E_y)\times(\pi_2(y)\times E_y)\big)\cap\overline{\text{Asy}_{\ba}(K,\Phi)}\Big)
 \\&\overset{\eqref{22-03-06-07}}=\pi_{X}\times \pi_{X}\Big(\big(\supp(\mu_y)\times\supp(\mu_y)\big)\cap\overline{\text{Asy}_{\ba}(K,\Phi) }\Big)
 \\&\supset\pi_{X}\times \pi_{X}\Big(\overline{\big(\supp(\mu_y)\times\supp(\mu_y)\big)\cap\text{Asy}_{\ba}(K,\Phi) }\Big)
 \\&\overset{\eqref{22-07-07-05}}= \pi_{X}\times \pi_{X}\big(\supp(\mu_y)\times\supp(\mu_y)\big)=E_y\times E_y.
 \end{align*}
  This shows that $\mathbf{P}(y)$ is a dense  subset of $E_{y}\times E_y$. 
 \end{proof}
 Denoting $\mathbf{C}(y)=\mathbf{P}(y) \cap \mathbf{D}_{\delta_y}(y)$, it is clear that   $\mathbf{C}(y)$ is a dense  $G_\delta$ subset of $E_{y}\times E_y$.  According to Lemma \ref{Myc}, there  exists a dense Mycielski subset $S_y\subseteq E_y$
such that 
$$S_y\times S_y\subseteq \mathbf{C}(y)\cup \Delta_{E_y},$$ 
where $\Delta_{E_y}=\{ (x, x): x\in E_y\}$. Obviously, $S_y\subseteq E_y\subseteq K_y$  and if $(x_1, x_2)$ is a pair of distinct points in  $S_y$, then $(x_1, x_2)\in \mathbf{C}(y)$ which implies that  
\begin{align*}
&\liminf_{N\to+\infty}\frac{1}{N}\sum_{i=1}^{N} d\big(\phi(a_i, \omega)x_1, \phi(a_i, \omega)x_2\big)=0,
\\&\limsup_{N\to+\infty}\frac{1}{N}\sum_{i=1}^{N} d\big(\phi(a_i, \omega)x_1, \phi(a_i, \omega)x_2\big)>\epsilon_0:=\delta_y/2.
\end{align*}
This finishes the proof of  Lemma \ref{22-03-06-08}.
\end{proof}

Since $\nu(\hat{Y})=1$, there exists $\hat{\Omega}\in\mathcal{F}$ with $\mathbb{P}(\hat{\Omega})=1$ such that for any $\omega\in\hat{\Omega}$, one has that 
$\pi_2^{-1}(\omega)\cap\hat{Y}\neq\varnothing$.  Taking some $y\in \pi_2^{-1}(\omega)\cap \hat{Y}$,  then $S(\omega):=S_y$ is the desired subsets.
\end{proof}

\subsection{Proof of Theorem \ref{main2}}\label{22-03-08-03}
In this subsection, we use the natural extension of the measure-preserving systems to deal with that $(X,\phi)$ is not injective. The reason why this method is not applicable to a general infinite positive integer sequence $\ba$  is that  the sequence $\ba$ will  arise a deviation for \eqref{22-03-09-02}. Note that the symbols in this subsection are independent of the symbols in Section \ref{22-04-22-03}.

\begin{proof}[Proof of Theorem \ref{main2}]
 Assume that
$ h_{top}(\phi,K)>0$. By Proposition \ref{VP} there exists $\mu\in \mathcal{E}_{\mathbb{P}}^K(\Omega\times X)$  such
that $h_\mu(\phi,K)>0$.  $\pi_\Omega:  (K,\mathcal{K},\mu,\Phi)\rightarrow (\Omega, \mathcal{F},\mathbb{P},\theta)$ is a factor map between two MDS on standard probability spaces. 

%Recall that the
%entropy of RDS $\phi$ with respect to $\mu$  is given by
%$$h_\mu(\phi):=h_\mu(\Phi|\pi_\Omega)=h_\mu(\Phi|\pi_\Omega^{-1}(\mathcal{F})).$$

 Let
$\Pi_{K}:\big(\bar{K},\bar{\mathcal{K}},\bar{\mu}, \bar{\Phi}\big)\rightarrow \big(K,\mathcal{K},\mu,\Phi\big)$ be
the  natural extension of $\big(K,\mathcal{K},\mu,\Phi\big)$ and  $\Pi_{\Omega}: (\bar{\Omega},
\bar{\mathcal{F}},\bar{\mathbb{P}},\bar{\theta})\rightarrow (\Omega, \mathcal{F},\mathbb{P},\theta)$ be the natural extension of $(\Omega, \mathcal{F},\mathbb{P},\theta)$. Define $\bar{\pi}:\bar{K}\rightarrow \bar{\Omega}$ by
$\bar{\pi}((\omega_{i},x_{i})_{i\in \mathbb{Z}})=(\omega_{i})_{i\in \mathbb{Z}}$ for $(\omega_{i},x_{i})_{i\in
\mathbb{Z}}\in\bar{K}$.   Then 
$$\bar{\pi}:\big(\bar{K},\bar{\mathcal{K}},\bar{\mu}, \bar{\Phi}\big)\rightarrow
(\bar{\Omega},\bar{\mathcal{F}},\bar{\mathbb{P}},\bar{\theta})$$ is a factor map between two MDSs on standard probability space, where $\bar{K}$ and $\overline{\Omega}$ are  Borel subsets of the Polish spaces $(\Omega\times X)^\mathbb{Z}$ and $\Omega^{\mathbb{Z}}$, respectively, and $\pi_{\Omega}\circ \Pi_{K}=\Pi_\Omega \circ \bar{\pi}$. 

%%%%%%%%%%%
%%%%%%%%%%%

We divide the remainder of the proof into three steps.

\medskip \noindent{\bf{Step  1.}} In this step, we're going to introduce some notations for our proof.

 By Lemma \ref{GAR} and  Lemma \ref{na=zero}, we have
\begin{align*}
 h_{\bar{\mu}}(\bar{\Phi}|\bar{\pi})&=h_{\bar{\mu}}(\bar{\Phi}|\bar{\pi})+h_{\bar{\mathbb{P}}}(\bar{\theta}|\Pi_\Omega)
\\ &=h_{\bar{\mu}}(\bar{\Phi}|\Pi_\Omega\circ\bar{\pi})=h_{\bar{\mu}}(\bar{\Phi}|\pi_{\Omega}\circ \Pi_{K})
\\ &=h_{\bar{\mu}}(\bar{\Phi}|\Pi_{K})+h_\mu(\Phi|\pi_{\Omega})=h_\mu(\Phi|\pi_{\Omega})=h_\mu(\phi,K)>0.
\end{align*}
Let $\mathcal{P}_{\bar{\mu}}(\bar{\pi}^{-1}\bar{\mathcal{F}})$ be the relative Pinsker $\sigma$-algebra of
$\big(\bar{K},\bar{\mathcal{K}},\bar{\mu}, \bar{\Phi}\big)$ with respect   to $(\bar{\Omega},\bar{\mathcal{F}},\bar{\mathbb{P}},\bar{\theta})$. Therefore, there exists a MDS  $(Y,\mathcal{Y},\nu,S)$ on standard probability space and  two factor maps
$$\pi_1:\big(\bar{K},\overline{\mathcal{K}},\bar{\mu}, \bar{\Phi}\big)\rightarrow (Y,\mathcal{Y},\nu,S),\
\pi_2:(Y,\mathcal{Y},\nu,S)\rightarrow  (\bar{\Omega},\bar{\mathcal{F}},\bar{\mathbb{P}},\bar{\theta})$$ between two MDSs on standard probability space such that  $\pi_2\circ \pi_1=\bar{\pi}$  and
$\pi_1^{-1}(\mathcal{Y})=\mathcal{P}_{\bar{\mu}}(\bar{\pi}^{-1}\bar{\mathcal{F}})\pmod{\bar{\mu}}$. That is, $\pi_1:\big(\bar{K},\bar{\mathcal{K}},\bar{\mu}, \bar{\Phi}\big)\to(Y,\mathcal{Y},\nu,S)$ 
is the relative Pinsker factor map with respect  to $\bar{\pi}^{-1}\bar{\mathcal{F}}$.

 Define metrics $\rho$, $\rho_1$ and $\rho_2$ on $ X^{\mathbb{Z}}$, $\Omega^{\z}$ and $(\Omega\times X)^{\z}$, respectively as follows:
\begin{align*}
&\rho(\vec{x}_1, \vec{x}_2)=\sum_{i\in  \mathbb{Z}} \frac{1}{2^{|i|}}
\frac{d(x_{i}^1,x_{i}^2)}{1+d(x_{i}^1,x_{i}^2)},\quad \rho_1( \vec{\omega}_1,  \vec{\omega}_2)=\sum_{i\in  \mathbb{Z}} \frac{1}{2^{|i|}}
\frac{d_{\Omega}(\omega_{i}^1,\omega_{i}^2)}{1+d_{\Omega}(\omega_{i}^1,\omega_{i}^2)} 
\\&\rho_2(\vec{\boldsymbol{k}}_1, \vec{\boldsymbol{k}}_2)=\max\{\rho(\vec{x}_1, \vec{x}_2),\rho_1(\vec{\omega}_1,  \vec{\omega}_2)\},
\end{align*}
where  $\vec{x}_j=(x^j_{i})_{i\in \mathbb{Z}}\in X^{\mathbb{\z}}$, $\vec{\omega}_j=(\omega^j_{i})_{i\in  \mathbb{Z}}\in\Omega^{\z}$, $\vec{\boldsymbol{k}}_j=(\omega^j_{i},x^j_{i})_{i\in
\mathbb{Z}}\in (\Omega\times X)^{\z}$ for $j=1,2$,  and $d$, $d_{\Omega}$ are the compatible complete metrics on $X$, $\Omega$, respectively.
For simplicity,  we sometimes identify $(\Omega\times  X)^{\z}$ with
$\Omega^{\z}\times X^{\z}$ by
$$((\omega_{i},x_{i}))_{i\in \mathbb{Z}}\in (\Omega\times
X)^{\z}\simeq ((\omega_{i})_{i\in \mathbb{Z}}, (x_{i})_{i\in \mathbb{Z}})\in \Omega^{\z}\times
X^{\z}.$$ 

\iffalse
\begin{claim}
 For $n\in  \mathbb{N}$,  let  $B_n=\{ (\vec{k}_1, \vec{k}_2)\in  \bar{K}\times \bar{K} :\rho(\vec{k}_1,\vec{k}_2)\geq \frac{1}{n}\}, $
 then $\bar{\mu}\times_Y \bar{\mu}(\bigcup \limits_{n=1}^\infty B_n)=1$.
\end{claim}
\begin{proof}  Denoting  $$\Delta=\{(\vec{k}_1,\vec{k}_2)\in \bar{K}\times \bar{K}: \overline{\Pi}_{\bar{K}, a_i}\vec{k}_1=  \overline{\Pi}_{\bar{K}, a_i}\vec{k}_2 \text{ for each }i\in
\mathbb{Z}\},$$ then $\bar{K}\times \bar{K}\setminus \Delta=\bigcup_{n=1}^\infty B_n.$ Let $\Delta_{\bar{K}}=\{ (\bar{k},\bar{k}):\bar{k}\in \bar{K}\}$. For $\nu$-a.s. $y\in Y$, one has 
\begin{align*}
\bar{\mu}_y\times \bar{\mu}_y(\Delta)&=\bar{\mu}_y\times\bar{\mu}_y\left(\big(\bar{\pi}^{-1}(\pi_2(y))\times \bar{\pi}^{-1}(\pi_2(y)\big)\cap \Delta\right)
\\&=\bar{\mu}_y\times\bar{\mu}_y(\Delta_{\bar{K}})
\\&=\int_{\bar{K}}\left(
\int_{\bar{K}}1_{\Delta_{\bar{K}}}(\vec{k}_1,\vec{k}_2) d\bar{\mu}_y(\vec{k}_2) \right)d\bar{\mu}_y(\vec{k}_1)
\\&\tag{By \eqref{22-03-03-01} in \textbf{Step 1}}=\int_{\bar{K}}\bar{\mu}_y(\vec{k}_1) d \bar{\mu}_y(\vec{k}_1)=0
\end{align*}
which implies $\bar{\mu}\times_Y \bar{\mu}(\bigcup \limits_{n=1}^\infty B_n)=1$. 
\end{proof}
\fi
\medskip \noindent{\bf{Step 2.}} In this step, we obtain a  similar result as Lemma \ref{22-03-06-08} for the natural extension system.

 Given $\delta>0$ and $\vec{\boldsymbol{k}}_0\in\bar{K}$, put
\begin{equation*}
\mathcal{W}(\bar{K}, \delta)=\{(\vec{\boldsymbol{k}}_1,\vec{\boldsymbol{k}}_2)\in \bar{K}\times \bar{K}: \limsup_{N\rightarrow+\infty}
 \frac{1}{N}\sum_{i=1}^{N} \rho_2((\bar{\Phi}\times\bar{\Phi})^{i}(\vec{\boldsymbol{k}}_1,\vec{\boldsymbol{k}}_2))>\delta\},
\end{equation*} 
and 
$$W^s(\vec{\boldsymbol{k}}_0,\bar{\Phi})=\{\vec{\boldsymbol{k}}\in\bar{K}: \lim_{n\to+\infty}\rho_2(\bar{\Phi}^{n}\vec{\boldsymbol{k}},\bar{\Phi}^{n}\vec{\boldsymbol{k}}_0)=0\}.$$
Let  $\bar{\mu}=\int_Y \bar{\mu}_y d \nu(y)$ be the disintegration of $\bar{\mu}$ relative to the  factor
$(Y,\mathcal{Y},\nu,S)$.   By  Lemma \ref{key-lem}, Theorem \ref{22-03-06-05} and Theorem \ref{infinity}, there is a $\nu$-full measure  set  $\hat{Y}$  with $\nu(\hat{Y})=1$ such that for  $y\in \hat{Y}$ 
with following properties:
\begin{enumerate}[(P1)]
\item Denoting $\vec{\omega}(y)=(\omega_i(y))_{i\in\z}:=\pi_2(y)$,  then $\phi(n,\omega_i):X\rightarrow X $ is  continuous and $K(\omega_i(y))$ is compact subset of $X$  for each  $n\in \mathbb{N}$ and $i\in\z$.
 \item\label{22-04-22-01-04} $\bar{\mu}_y(\pi_2(y)\times\bar{K}_y)=1$ and $\bar{\mu}_y$ is non-atomic where $\bar{K}_y:=\{\vec{x}\in \Pi_{i\in\mathbb{Z}}K(\omega_i(y)): \phi(1, \omega_i)x_i=x_{i+1}\text{ for any } i\in\mathbb{Z}\}$.
 \item There exists $\delta_y>0$ such that  $\mu_y\times \mu_y(\mathcal{W}(\bar{K},\delta_y))=1.$
 \item   For any $\vec{\boldsymbol{k}}\in\bar{\pi}_1^{-1}(y)$, one has that  $\overline{W^{s}(\vec{\boldsymbol{k}},\bar{\Phi})\cap \supp(\bar{\mu}_{y})}=\supp(\bar{\mu}_ y).$

\end{enumerate}
With the similar argument as Lemma \ref{22-03-06-08}, we have following lemma.
\begin{lem}\label{22-04-22-08}
For any $y\in\hat{Y}$, $\bar{K}_{y}$ is a compact subset of
$(X^{\z},\rho)$ and there exists a Mycielski subset $\bar{S}_y$ of  $\bar{K}_y$ and  $\epsilon_0>0$ such that for any distinct two points $\vec{x}_1, \vec{x}_2\in \bar{S}_y$ satisfies:
  \begin{align}\label{22-07-07-07}
&\liminf_{N\to+\infty}\frac{1}{N}\sum_{i=1}^{N} \rho_2\big(\bar{\Phi}^{i}(\vec{\omega}(y), \vec{x}_1), \bar{\Phi}^{i}(\vec{\omega}(y),\vec{x}_2)\big)=0,
\\&\label{22-07-07-08}\limsup_{N\to+\infty}\frac{1}{N}\sum_{i=1}^{N} \rho_2\big(\bar{\Phi}^{i}(\vec{\omega}(y),\vec{x}_1), \bar{\Phi}^{i}(\vec{\omega}(y), \vec{x}_2)\big)>\epsilon_0,
\end{align}
where  $\vec{\omega}(y)=(\omega_i(y))_{i\in\z}:=\pi_2(y)$.
 \end{lem}

\medskip \noindent{\bf Step  3.} In this step, we finish the proof of Theorem \ref{main2}.

Recall that $$\pi_2:(Y,\mathcal{Y},\nu,S)\rightarrow (\overline{\Omega},\bar{\mathcal{F}}, \bar{\mathbb{P}},\bar\theta) \text{ and }\Pi_\Omega:(\overline{\Omega},\bar{\mathcal{F}}, \bar{\mathbb{P}},\bar\theta) \rightarrow(\Omega,\mathcal{F}, \mathbb{P},\theta),$$  where $\pi_2$ is a factor map between two MDSs and $\Pi_\Omega$ is the projection. Since $\nu(\hat{Y})=1$, by Lemma \ref{factor}, there exists  $\hat{\Omega}\in\mathcal{F}$ with $\mathbb{P}(\hat{\Omega})=1$ and $(\Pi_\Omega\circ
\pi_2)^{-1}(\omega)\cap \hat{Y}\neq \varnothing$ for  each $\omega\in \hat{\Omega}$.

From now,  fix $\omega\in \hat{\Omega}$ and  $y\in \hat{Y}$ such that $\Pi_\Omega\circ \pi_2(y)=\omega$. Then we have  $\omega=\omega_{0}(y)$. By Lemma \ref{22-04-22-08},  there exist $\epsilon_0>0$ and a Mycielski subset $\bar{S}_y$ of  $\bar{K}_y$ such that for each pair $(\vec{x}_1, \vec{x}_2)$ of distinct points in  $\bar{S}_y$, satisfies  \eqref{22-07-07-07} and \eqref{22-07-07-08}.  Let $\eta:\bar{K}_y\rightarrow K_{\omega}$ be the natural projection of coordinate with $\eta(\vec{x})=x_{0}$ for $\vec{x}=(x_{i})_{i\in \mathbb{Z}}\in  X^{\z}$. Put $S_\omega=\eta(\bar{S}_y)$. Then
$S_\omega\subseteq K_{\omega}$. 

In the following we show  that $S_\omega$ is a Mycielski chaotic set for $(\omega,\phi)$. Firstly we claim that map $\eta: \bar{S}_y\rightarrow S_\omega$ is injective. If this is not true,
then there  exist two distinct points $\vec{x}_1,\vec{x}_2$ in $\bar{S}_y$ such
that $\eta(\vec{x}_1)=\eta(\vec{x}_2)$,  i.e.  $x_{0}^1=x_{0}^2$. Since $\vec{x}_1, \vec{x}_2\in \bar{K}_y$, we  have
\begin{align}\label{eq-n}
x_{i}^1=\phi(i, {\omega_{0}(y)})x_0^1=\phi(i, {\omega_{0}(y)})x_0^2=x_{i}^2 \end{align}
 for each $i\in \mathbb{N}_0$.  Thus
\begin{equation*}
\begin{split}
&\lim_{n\rightarrow+\infty} \rho((\phi(n,{\omega_{i}(y)})x_{i}^1)_{i\in \mathbb{Z}},
(\phi(n, {\omega_{i}(y)})x_{i}^2)_{i\in \mathbb{Z}})=\lim_{n\rightarrow+\infty} \rho((x_{i+n}^1)_{i\in\mathbb{Z}},(x_{i+n}^2)_{i\in\mathbb{Z}})\overset{\eqref{eq-n}}=0
\end{split}
\end{equation*}  
which contradicts \eqref{22-07-07-07}. Hence $\eta: \bar{S}_y\rightarrow S_\omega$ is injective. Since $\bar{S}_y$ is a Mycielski set, $\bar{S}_y=\bigcup_{j\in \mathbb{N}}C_j$  where each $C_j$ is a Cantor set. Since
$\eta:(C_j,\rho)\rightarrow (\eta(C_j),d)$ is a one to one surjective continuous map and $C_j$ is a compact
subset of $(X^{\z},\rho)$, it follows that $\eta:C_j\rightarrow \eta(C_j)$ is a homeomorphism. Thus $\eta(C_j)$ is a
Cantor set. Hence $S_\omega=\bigcup_{j\in \mathbb{N}} \eta(C_j)$ is a Mycielski set of $K_\omega$.

Given a pair $(x_1, x_2)$ of distinct points in  $S_{\omega}$,    one has that 
\begin{align*}
&\liminf_{N\to+\infty}\frac{1}{N}\sum_{j=1}^{N}d(\phi(j,\omega)x_1, \phi(j,\omega)x_2)
\\\leq&(1+d(x_1, x_2))\liminf_{N\to+\infty}\frac{1}{N}\sum_{j=1}^{N}\rho((\phi(j,\omega_{i}(y))x_{i}^1)_{i\in\mathbb{Z}}, (\phi(j,\omega_{i}(y))x_{i}^2)_{i\in\mathbb{Z}})=0,
\end{align*}
where $(x_i^k)_{i\in\mathbb{Z}}=\eta^{-1}(x_k)$ for $k=1,2$. This implies that \eqref{22-03-09-01} holds. On the other hand, recall that $\epsilon_0$ is the constant in Lemma \ref{22-04-22-08}.  Take $L\in\mathbb{N}$ such that  $\frac{1}{2^{L}}<\epsilon_0/6$. By 
$$\limsup_{N\to+\infty}\frac{1}{N}\sum_{k=1}^{N}\rho((\phi(k,\omega_{i}(y))x_{i}^1)_{i\in\mathbb{Z}}, (\phi(k,\omega_{i}(y))x_{i}^2)_{i\in\mathbb{Z}})\geq\epsilon_0,$$
there exists a sequence of positive integers $\{N_j\}_{j\in\mathbb{N}}$ with $L<N_1<N_2<...$ such that 
$$\frac{1}{N_j}\sum_{k=1}^{N_j}\rho((\phi(k,\omega_{i}(y))x_{i}^1)_{i\in\mathbb{Z}}, (\phi(k,\omega_{i}(y))x_{i}^2)_{i\in\mathbb{Z}})\geq\epsilon_0/2.$$
 Note that  for each $k\in\mathbb{N}$, 
\begin{align*}
(\phi(k, \omega_{i}(y))x_{i}^1)_{i\in\mathbb{Z}}=(x^1_{k+i})_{i\in\mathbb{Z}}\quad\text{and}\quad(\phi(k, \omega_{i}(y))x_{i}^2)_{i\in\mathbb{Z}}=(x^2_{k+i})_{i\in\mathbb{Z}}.
\end{align*}
It follows that 
\begin{align*}
&\frac{1}{N_j}\sum_{k=1}^{N_j}\sum_{|i|\leq L}\frac{d(x^1_{i+k}, x_{i+k}^2)}{2^i(1+d(x^1_{i+k}, x_{i+k}^2))}\geq\frac{\epsilon_0}{2}-\frac{1}{N_j}\sum_{k=1}^{N_j}\sum_{|i|\geq L+1}\frac{d(x^1_{i+k}, x_{i+k}^2)}{2^i(1+d(x^1_{i+k}, x_{i+k}^2))}\geq\frac{\epsilon_0}{2}-\frac{\epsilon_0}{3}=\frac{\epsilon_0}{6}.
\end{align*}
For each $j\in\mathbb{N}$, there exists $i_j\in [-L,L]$ such that 
$$\sum_{k=1}^{N_j}\frac{d(x^1_{i_j+k},x^1_{i_j+k})}{1+d(x^1_{i_j+k},x^1_{i_j+k})}=\max_{|i|\leq L}\left\{\sum_{k=1}^{N_j}\frac{d(x^1_{i+k},x^1_{i+k})}{1+d(x^1_{i+k},x^1_{i+k})}\right\}.$$
Therefore,
\begin{align*}
\notag\frac{1}{N_j}\sum_{k=1}^{N_j}d(x^1_{i_j+k}, x^2_{i_j+k})&\geq\frac{1}{N_j}\sum_{k=1}^{N_j}\frac{d(x^1_{i_j+k}, x^2_{i_j+k})}{1+d(x^1_{i_j+k}, x^2_{i_j+k})}
\\&\geq\frac{1}{3}\frac{1}{N_j}\sum_{k=1}^{N_j}\sum_{|i|\leq L}\frac{d(x^1_{i+k}, x_{i+k}^2)}{2^i(1+d(x^1_{i+k}, x_{i+k}^2))}\geq \frac{\epsilon_0}{18},
\end{align*}
which  implies that  
\begin{align*}
\limsup_{N\to\infty}\sum_{i=1}^{N}d(\phi(i,\omega)x_1, \phi(i,\omega)x_2)&\geq\limsup_{j\to\infty}\frac{1}{N_j+L}\sum_{i=1}^{N_j+L}d\big(\phi(i,\omega)x_1, \phi(i,\omega)x_2\big)
\\&\geq\limsup_{j\to\infty}\frac{N_j}{N_j+L}\frac{\epsilon_0}{18}>\frac{\epsilon_0}{36}.
\end{align*}
This finishes the proof of Theorem \ref{main2}.
\end{proof}

\noindent{\bf Acknowledgments.}
The authors would like to thanks   Professor W. Huang for kindly suggestions  and Associate Professor L. Xu  for providing a brief proof of Proposition \ref{21-03-06-02}.  C. Liu and J. Zhang are partially supported by NNSF of China (12090012).  F. Tan is supported by NSF of China (11871228).

%\bibliographystyle{elsarticle-num}
%\bibliography{ref}

\end{document}